\documentclass{article}

\usepackage{savesym}
\usepackage[disable]{todonotes}
\usepackage{amsthm,amsmath}
\usepackage{amssymb,latexsym,graphicx}
\usepackage{amscd}
\usepackage[all,cmtip]{xy}
\usepackage{tikz-cd}

\usepackage{accents}
\usepackage{cite}
\usepackage{mathtools}
\usepackage{excludeonly}
\usepackage{stmaryrd} 

\usepackage{calligra}
\DeclareMathAlphabet{\mathcalligra}{T1}{calligra}{m}{n}
\DeclareMathAlphabet{\mathpzc}{OT1}{pzc}{m}{it}
\usepackage{hycolor}
\usepackage{xcolor}
\usepackage[
                       colorlinks=true,
                       linkcolor=black, 
                       citecolor=black, 
                       urlcolor=blue,
                     ]{hyperref}

\usepackage[intoc]{nomentbl}
\makenomenclature

\usepackage{stackengine} 

\newtheorem{theorem}{Theorem}[section]
\newtheorem{corollary}[theorem]{Corollary}

\newtheorem{lemma}[theorem]{Lemma}
\newtheorem{proposition}[theorem]{Proposition}

\theoremstyle{definition}
\newtheorem{definition}[theorem]{Definition}

\newtheorem{remark}[theorem]{Remark}

\newtheorem{example}[theorem]{Example}

\theoremstyle{remark}

\newcommand{\D}{{\mathbb{D}}}
\newcommand{\F}{{\mathbb{F}}}

\newcommand{\N}{{\mathbb{N}}}

\newcommand{\Q}{{\mathbb{Q}}}
\newcommand{\R}{{\mathbb{R}}}
\renewcommand{\SS}{{\mathbb{S}}}

\newcommand{\Z}{{\mathbb{Z}}}
\newcommand{\Dd}{{\mathcal{D}}}

\newcommand{\Gg}{{\mathcal{G}}}   

\newcommand{\Nn}{{\mathcal{N}}}
\newcommand{\Oo}{{\mathcal{O}}}

\newcommand{\Ss}{{\mathcal{S}}}
\newcommand{\Tt}{{\mathcal{T}}}
\newcommand{\Uu}{{\mathcal{U}}}

\newcommand{\Ww}{{\mathcal{W}}}

\newcommand{\im}{{\rm im\, }}             
\newcommand{\id}{{\rm id}}                

\newcommand{\IND}{{\rm ind}}                  
\newcommand{\Fix}{{\rm Fix}}          
\newcommand{\Crit}{{\rm Crit}}        
\newcommand{\Ho}{{\rm H}}              

\newcommand{\I}{{\rm I}}               
\newcommand{\RP}{\R{\rm P}}            
\newcommand{\cat}{\mathrm{cat}}            
\newcommand{\catamb}{\mathrm{cat^{a}}}     
\newcommand{\cupp}{\mathrm{cup}}           
\newcommand{\sub}{\mathrm{sub}}            
\newcommand{\CUP}{\mathop{\cup}}           
\newcommand{\CAP}{\mathop{\cap}}           
 
\newcommand{\interior}[1]{\accentset{\circ}{#1}} 
\newcommand{\boundary}[1]{\dot{#1}}              
\newcommand{\norm}{{\rm norm}}

\newcommand{\eps}{{\varepsilon}}

\newcommand{\mbf}[1]{\text{\boldmath $#1$}}  

\newcommand{\ba}{{\mbf{a}}}
\newcommand{\bb}{{\mbf{b}}}
\newcommand{\bc}{{\mbf{c}}}

\newcommand{\balpha}{{\mbf{\alpha}}}
\newcommand{\bbeta}{{\mbf{\beta}}}

\def\NABLA#1{{\mathop{\nabla\kern-.5ex\lower1ex\hbox{$#1$}}}}
\def\Nabla#1{\nabla\kern-.5ex{}_{#1}}
\def\Tabla#1{\Tilde\nabla\kern-.5ex{}_{#1}}
\def\abs#1{\mathopen|#1\mathclose|}   
\def\Abs#1{\left|#1\right|}            
\def\norm#1{\mathopen\|#1\mathclose\|}

\renewcommand{\Tilde}{\widetilde}

\newcommand{\p}{{\partial}}

\newcommand{\Index}[1]{#1\index{#1}}

\newlength\eqshift
\renewcommand\theequation{\thesection.\arabic{equation}}
\let\savetheequation\theequation

\makeatletter
\renewcommand*\env@matrix[1][\arraystretch]{%
  \edef\arraystretch{#1}%
  \hskip -\arraycolsep
  \let\@ifnextchar\new@ifnextchar
  \array{*\c@MaxMatrixCols c}}
\makeatother

\makeatletter
\let\save@mathaccent\mathaccent
\newcommand*\if@single[3]{%
  \setbox0\hbox{${\mathaccent"0362{#1}}^H$}%
  \setbox2\hbox{${\mathaccent"0362{\kern0pt#1}}^H$}%
  \ifdim\ht0=\ht2 #3\else #2\fi
  }
\newcommand*\rel@kern[1]{\kern#1\dimexpr\macc@kerna}
\newcommand*\widebar[1]{\@ifnextchar^{{\wide@bar{#1}{0}}}{\wide@bar{#1}{1}}}
\newcommand*\wide@bar[2]{\if@single{#1}{\wide@bar@{#1}{#2}{1}}{\wide@bar@{#1}{#2}{2}}}
\newcommand*\wide@bar@[3]{%
  \begingroup
  \def\mathaccent##1##2{%
    \let\mathaccent\save@mathaccent
    \if#32 \let\macc@nucleus\first@char \fi
    \setbox\z@\hbox{$\macc@style{\macc@nucleus}_{}$}%
    \setbox\tw@\hbox{$\macc@style{\macc@nucleus}{}_{}$}%
    \dimen@\wd\tw@
    \advance\dimen@-\wd\z@
    \divide\dimen@ 3
    \@tempdima\wd\tw@
    \advance\@tempdima-\scriptspace
    \divide\@tempdima 10
    \advance\dimen@-\@tempdima
    \ifdim\dimen@>\z@ \dimen@0pt\fi
    \rel@kern{0.6}\kern-\dimen@
    \if#31
      \overline{\rel@kern{-0.6}\kern\dimen@\macc@nucleus\rel@kern{0.4}\kern\dimen@}%
      \advance\dimen@0.4\dimexpr\macc@kerna
      \let\final@kern#2%
      \ifdim\dimen@<\z@ \let\final@kern1\fi
      \if\final@kern1 \kern-\dimen@\fi
    \else
      \overline{\rel@kern{-0.6}\kern\dimen@#1}%
    \fi
  }%
  \macc@depth\@ne
  \let\math@bgroup\@empty \let\math@egroup\macc@set@skewchar
  \mathsurround\z@ \frozen@everymath{\mathgroup\macc@group\relax}%
  \macc@set@skewchar\relax
  \let\mathaccentV\macc@nested@a
  \if#31
    \macc@nested@a\relax111{#1}%
  \else
    \def\gobble@till@marker##1\endmarker{}%
    \futurelet\first@char\gobble@till@marker#1\endmarker
    \ifcat\noexpand\first@char A\else
      \def\first@char{}%
    \fi
    \macc@nested@a\relax111{\first@char}%
  \fi
  \endgroup
}
\makeatother

\long\def\symbolfootnote[#1]#2{\begingroup%
\def\thefootnote{\fnsymbol{footnote}}\footnote[#1]{#2}\endgroup}

\title{Conley pairs in geometry -- Lusternik-Schnirelmann theory and more}

\author{
  Joa Weber\footnote{
        {\bf Financial support:}
        O presente trabalho foi realizado com apoio da FAPESP e do CNPq,
        Conselho Nacional de Desenvolvimento Cient\'{\i}fico e Tecnol\'{o}gico - Brasil.
        \newline
         {\bf Email:} joa@ime.unicamp.br
        {\bf Address:}
        Instituto de Matem\'{a}tica, Estat\'{\i}stica
        e Computa\c{c}\~{a}o Scient\'{\i}fica,
        Universidade Estadual de Campinas,
        Rua S\'{e}rgio Buarque de Holanda~651,
        Campinas, SP, Brasil. 
  }
    \\
    IMECC UNICAMP
}

\date{\today}
\begin{document}

\maketitle 

\begin{abstract}
Firstly, we wish to motivate that Conley pairs,
realized via Salamon's definition~\cite{salamon:1990a},
are rather useful building blocks in geometry:
Initially we met Conley pairs
in an attempt to construct Morse
filtrations of free loop spaces~\cite{weber:2014c}.
From this fell off quite naturally, firstly, an alternative proof~\cite{Weber:2014e}
of the cell attachment theorem in Morse theory~\cite{milnor:1963a}
and, secondly, some ideas~\cite{Majer:2015a}
how to try to organize the closures of the unstable
manifolds of a Morse-Smale gradient flow as a CW decomposition
of the underlying manifold.
Relaxing non-degeneracy of critical points to isolatedness
we use these Conley pairs to implement the gradient flow proof
of the Lusternik-Schnirelmann Theorem~\cite{Lusternik:1934a}
proposed in Bott's survey~\cite{bott:1982b}.

Secondly, we shall use this opportunity to provide an exposition of
Lusternik-Schnirelmann (LS) theory 
\todo{work in rational LS thy:\cite{Felix:2001a}}
based on thickenings
of unstable manifolds via Conley pairs. We shall cover the
Lusternik-Schnirelmann Theorem~\cite{Lusternik:1934a},
cuplength, subordination,
the LS refined minimax principle, and a variant of the LS category
called ambient category.
\end{abstract} 
 
\tableofcontents



\section{Introduction and applications}
Throughout let $\varphi$ be a downward gradient flow 
on a (smooth) closed manifold~$M$, that is a one-parameter
group $\{\varphi_t\}_{t\in\R}$ of diffeomorphisms of $M$
determined by
$$
     \frac{d}{dt}\varphi_t=-\left(\nabla f\right)\circ\varphi_t,\qquad
     \varphi_0=\id,
$$
for some function $f:M\to\R$ of class $C^2$ and where the
gradient $\nabla f=\nabla^g f$ is determined by the identity
$df=g(\nabla f,\cdot)$, given a Riemannian metric $g$ on $M$.
By $\nabla$ we shall also denote the Levi-Civita connection
associated to $g$.

While Conley theory~\cite{conley:1978a} deals with rather general
flow invariant sets, the present paper concentrates on the two
simplest cases, that of only isolated and that of only non-degenerate
critical points.
\begin{itemize}
\item
  \textit{Morse theory}: All critical points $x\in\Crit f$ are non-degenerate
  in the sense that the Hessian of $f$ at $x$ is non-singular, that is
  zero is not an eigenvalue. The Morse index $\IND_f(x)$ of a critical
  point is the number $k$ of negative eigenvalues, counted with multiplicities.
  Such $f$ is called a Morse function and satisfies the (weak) Morse
  inequalities: There is a lower bound for the number of critical points
  of $f$ of Morse index $k$ in terms of the dimension of the $k^{\rm th}$
  singular homology of $M$ with coefficients in a field $\F$. Namely,
  $$
     \Abs{\Crit_k f}=:c_k\ge\beta_k(\F):=\dim\Ho_k(M;\F)
  $$
  for every integer $k=0,\dots,n:=\dim M$; see e.g.~\cite{milnor:1963a}.
  For rational coefficients the integer $\beta_k(\Q)$ is
  the $k^{\rm th}$ \textbf{Betti number} $b_k(M)$ of $M$.
\item
  \textit{Lusternik-Schnirelmann (LS) theory}: All critical
  points of $f$ are isolated. Then their number is bounded below by
  another homotopy invariant, the \textbf{Lusternik-Schnirelmann category}
  $\cat(M)$. It is the least integer $\ell\ge 1$ such that there is
  an open cover $U_1,\dots,U_\ell$ of $M$ by $\ell$
  nullhomotopic\footnote{
    A subset $A\subset M$ is \textbf{nullhomotopic},
    or \textbf{contractible in \boldmath$M$},
    if the inclusion $A\hookrightarrow M$ is nullhomotopic:
    $\exists F\in C^0([0,1]\times A,M)$
    $\exists m\in M$:
    $F(0,a)=a$ and $F(1,a)=m$ $\forall a\in A$.
    }
  sets.
\end{itemize}

Note that non-degenerate implies isolated by the inverse function theorem.

\begin{theorem}[Lusternik-Schnirelmann~\cite{Lusternik:1934a}]\label{thm:LS}
Suppose $f$ is a $C^2$ function on a closed manifold $M$, then
$$
     \Abs{\Crit f}\ge \cat(M).
$$
\end{theorem}

Pick a Morse function to get finiteness of $\cat(M)$.
Palais~\cite{Palais:1966b} generalized Theorem~\ref{thm:LS}  to
infinite dimensions replacing compactness of $M$ 
by the Palais-Smale condition, also called condition~(C).
In the present exposition we complete, with the help of Conley
pairs, the following gradient flow proof of the Lusternik-Schnirelmann
Theorem~\ref{thm:LS} proposed in Bott's survey~\cite[p.\,342]{bott:1982b}.

\begin{proof}
Assume $\Crit f$ is a finite set; otherwise, we are done.
Pick a Riemannian metric on $M$ and consider the downward
gradient flow $\varphi=\{\varphi_t\}_{t\in\R}$.
The stable manifold $W^s(x)$ of a
critical point $x$ is the set of all $q\in M$
for which the limit $\lim_{t\to\infty}\varphi_t q$
exists and is equal to $x$. The limit exists for every point $p$ of
$M$ due to \emph{compactness} of $M$ and \emph{isolatedness}
of the critical points (the set $\Crit f$ is finite).\footnote{
  The $\omega$-limit set $\omega(p)$ is a connected subset of $\Crit f$,
  thus a singleton $\{y\}$, as $\Crit f$ is discrete;
  see e.g.~\cite[Ch.\,1 \S 1 Ex.\,3]{palis:1982a}.
  Hence $\lim_{t\to\infty}\varphi_t p=\omega(p)=y\in\Crit f$.
  }
Therefore the stable manifolds cover $M$.
While a stable manifold in general is not an open subset of $M$
and, without the non-degeneracy assumption
on its critical point $x$, also not necessarily
any more an embedded open disk,\footnote{
  For $f(u)=u^3$ with $u\in\R$ the stable manifold
  of $0$ is the ``half disk'' $W^u(0)=[0,\infty)$.
  }
it still contracts onto $x$.
The yet missing piece is Proposition~\ref{prop:amb-thick}~(i)
which asserts that one can thicken each stable manifold to an open
subset $\Ww_x^*$ of $M$ preserving contractibility. So $M$ is covered
by $\ell=\abs{\Crit f}$ open nullhomotopic sets.
\end{proof}

The proof of Proposition~\ref{prop:amb-thick} (existence of thickening)
rests on the notion of

\subsubsection*{Conley pairs}
A basic notion in Conley theory is that of an
index pair for an isolated invariant set $S$.
In the Morse case an explicit construction for $S=\{ x\}$
has been given by Salamon~\cite{salamon:1990a}: For $x\in\Crit f$
and reals $\eps,\tau>0$ define a pair of spaces $(N,L)$ by
\begin{equation}\label{eq:N-iso}
\begin{split}
     N=N_x^{\eps,\tau}
   :&=
     \left\{ p\in M\mid
     \text{$f(p)\le c+\eps$, $f(\varphi_\tau p)\ge c-\eps$}
     \right\}_x ,\quad c:=f(x),
\end{split}
\end{equation}
where $\{\ldots\}_x$ denotes the path connected component
that contains $x$, and
\begin{figure}
\begin{minipage}[b]{.49\linewidth}
  \centering
  \includegraphics
                             {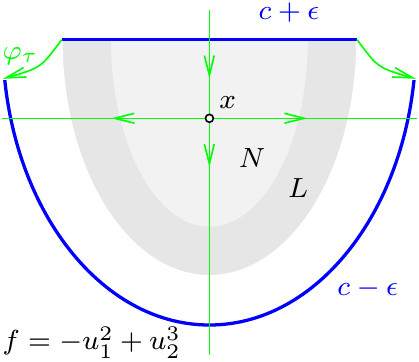}
  \caption{$(N,L)$ for isolated $x\in\Crit f$}
  \label{fig:fig-isolated}
\end{minipage}
\begin{minipage}[b]{.49\linewidth}
  \centering
  \includegraphics
                             {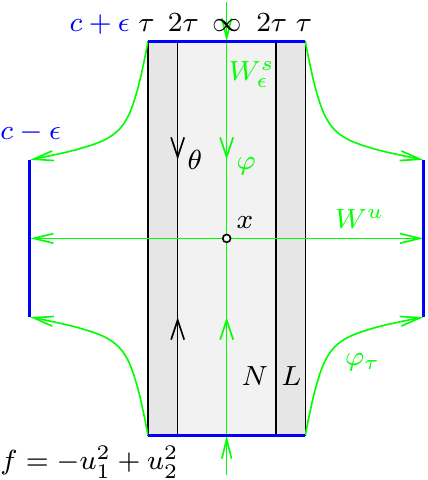}
  \caption{Non-degenerate $x$}
  \label{fig:fig-non-degenerate}
\end{minipage}
\end{figure}
\begin{equation}\label{eq:L-iso}
     L=L_x^{\eps,\tau}
     :=\{ p\in N\mid f(\varphi_{2\tau} p)\le c-\eps\}.
\end{equation}
By Sard's theorem we may suppose that $c\pm\eps$
are regular values of $f$; otherwise, perturb $\eps$.
Note that in case of a local minimum $x$ the set $N$ is a local
sublevel set and $L$ is empty (any point near $x$ eventually gets stuck
on the level $c$ of $x$, so none reaches the lower level $c-\eps$).
\\
In fact, for small $\eps$ and large $\tau$ it holds that
(i)~the fixed point $x$ of $\varphi$ lies in the interior of $N$ but
not in $L$, (ii)~there are no other fixed points in $N$,
(iii)~the subset $L$ is \emph{positively invariant} in $N$, and
(iv)~$L$ is an \emph{exit set} of $N$ in the sense that every forward flow
line which leaves $N$ runs through $L$ first; for details see
Definition~\ref{def:index-pair}.
For a proof of (i--iv) in the \emph{non-degenerate case} see~\cite{Weber:2015c};
see~\cite{weber:2014c} for an infinite dimensional context.

In the more general \emph{isolated case},
meaning that $x$ is just required to be an isolated critical point of $f$,
properties (i--iv) will be established in
Theorem~\ref{thm:Conley-pair} below.
Such $(N,L)$ is called a \textbf{Conley pair}, and $N$ a
\textbf{Conley block}, for the isolated critical point $x$.
Note that the part of the stable manifold $W^s=W^s(x)$ in
$N$ is the ascending disk $W^s_\eps=W^s_\eps(x):=W^s(x)\cap\{f\le f(x)+\eps\}$.
\newline
By the Shrinking Lemma~\ref{le:shrink-to-pt} one can fit $N$
into any given neighborhood of an isolated $x\in\Crit f$ by picking
$\eps,\tau>0$ sufficiently small and large, respectively.

\subsubsection*{Dynamical thickening -- non-degenerate case}
For \emph{non-degenerate} critical points $x$ much more can be shown
for small $\eps$ and large $\tau$: Firstly, the set $N=N_x^{\eps,\tau}$ contracts
onto the ascending disk $W^s_\eps$, as $\tau\to\infty$.
\begin{figure}
  \centering
  \includegraphics
                             {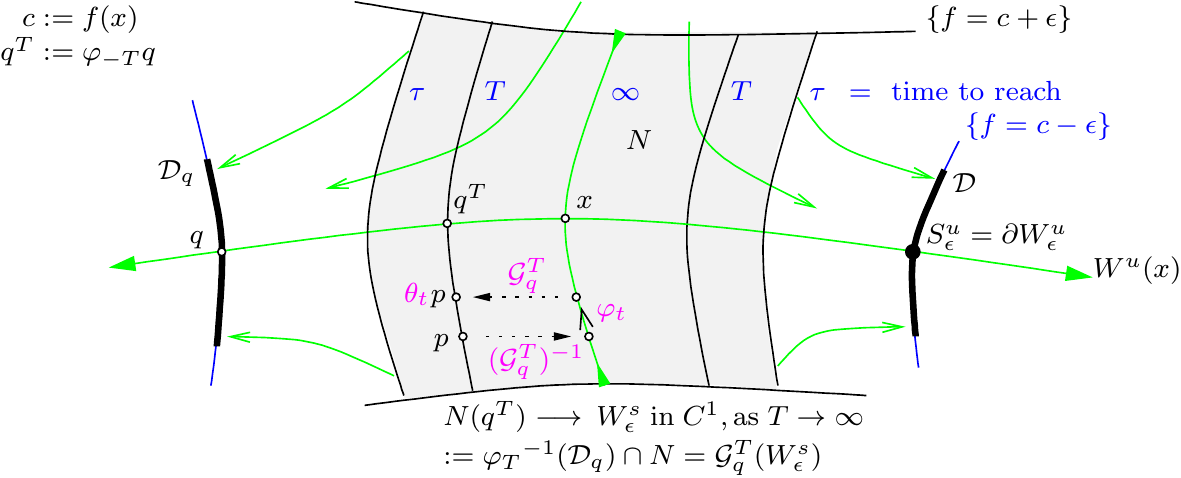}
  \caption{Dynamical thickening $(N,\theta)$ of
                 $(W^s_\eps,\varphi_{\ge 0})$}
  \label{fig:fig-dynam-thick}
\end{figure}
Secondly, the set $N$ is fibered by diffeomorphic copies of
$W^s_\eps$, one copy for each point of the part
$W^u_{\eps,\tau}:=N\cap W^u$ of the unstable manifold
$W^u=W^u(x)$~in~$N$; see Figure~\ref{fig:fig-dynam-thick}.
The construction of the fiber bundle
$W^s_\eps\hookrightarrow N\to W^u_{\eps,\tau}$
starts with a \emph{choice of fibers in the lower level set}:
Endow some neighborhood $\Dd$ of the
descending sphere $S^u_\eps=S^u_\eps(x):=W^u(x)\cap\{f=c-\eps\}$
in the level set $\{f=c-\eps\}$ with the structure
of a disk bundle $\D\hookrightarrow\Dd\to S^u_\eps$
where the codimension of the disk $\D$
is given by the Morse index $k=\IND_f(x)$.\footnote{
  E.g. pick a tubular neighborhood associated to the normal
  bundle of $S^u_\eps$ in $\{f=c-\eps\}$.
  }
For $q\in S^u_\eps$ and $T\ge\tau$ the fiber $N(q^T)$ over
$q^T:=\varphi_{-T}q$ by definition is the part in $N$ of the
pre-image ${\varphi_T}^{-1}\Dd_q$ of the fiber $\Dd_q$.
Let the fiber over $x$ be those points that never reach $\Dd$,
namely $N(x):=W^s_\eps$.
One shows~\cite{Weber:2015c,weber:2014c} that each fiber $N(q^T)$ can be written
as a $C^1$ graph over $W^s_\eps$ or, in other words,
as the image of a $C^1$ embedding $\Gg^T_q:W^s_\eps\to M$.
Transfer the forward semi-flow $\varphi_{\ge 0}$ on $W^s_\eps$
to each fiber via conjugation by the graph maps; see
Figure~\ref{fig:fig-dynam-thick}.
\begin{itemize}
\item
  \textit{Dynamical thickening $(N,\theta)$ of
    $(W^s_\eps,\varphi_{\ge 0})$.}
  As just described $N$ carries the structure of a fiber bundle
  $W^s_\eps\hookrightarrow N\to W^u_{\eps,\tau}$
  equipped with a fiberwise forward semi-flow $\theta$.
  Fibers and flow are modeled on the ascending disk $W^s_\eps$
  equipped with $\varphi_{\ge 0}$ and defined as
  graphs over $W^s_\eps$ and by conjugation.
  Hence $\theta$ deforms the total space $N$
  into the base space $ W^u_{\eps,\tau}:=N\cap W^u$;
  see Figures~\ref{fig:fig-non-degenerate} and~\ref{fig:fig-dynam-thick}.
\item
  \textit{Morse filtration.}
  Dynamical thickening was introduced in~\cite{weber:2014c} to
  construct a Morse filtration of a loop space in order to represent
  Morse homology for semi-flows in terms of singular homology.
  For an overview see~\cite{weber:2014b}.
\item
  \textit{Flow selector.} Dynamical thickening was applied
  in~\cite{Weber:2014e} to prove the cell
  attachment theorem in Morse theory~\cite[Thm.\,3.1]{milnor:1963a}
  through a basic two step deformation.
  Here entered crucially the construction of a \textbf{flow
  selector} $\Ss^+$, namely
  a hypersurface transverse to two flows,
  which came up in~\cite{Majer:2015a} in two flavors,
  via Conley pairs and via a carving technique.
  The point is that transversality allows to switch along $\Ss^+$
  from one to the other flow in a continuous fashion.
\item
  \textit{CW decomposition.} It is an old believe that the closures
  of the unstable manifolds of a Morse-Smale\footnote{
    The Morse-Smale condition requires transversality
    $W^u(x)\pitchfork W^s(y)$ for all $x,y\in\Crit f$.
    }
  gradient flow $\varphi$
  of a Morse function $f$ on a closed manifold $M$
  provide a CW decomposition of $M$.
  If the Riemannian metric is Euclidean near the critical points
  this is a result of Kalmbach~\cite{Kalmbach:1975a};
  see also Laudenbach in~\cite{Bismut:1992a}.
  In the general cases two methods of proof have been proposed
  in~\cite{2011arXiv1102.2838Q} and~\cite{2016arXiv161007509K}.
  
  Here is a geometric approach via asymptotic extensions
  of dynamical thickenings arising from joint ideas in~\cite{Majer:2006b,Majer:2015a}.
  But so far this only works in dimension two (where the
  problem of compatibly organizing fibers on overlaps is void).
  Whereas the flow $\varphi$ serves to identify
  diffeomorphically an unstable manifold
  with an open unit disk, this identification does not extend
  continuously to the boundary: In Figure~\ref{fig:fig-filling-curves-a}
  the endpoints of the $\color{green}\varphi$ flow lines
  do not even fill the boundary $\Delta$ of the unstable manifold $W$.
\begin{figure}
  \centering
  \includegraphics{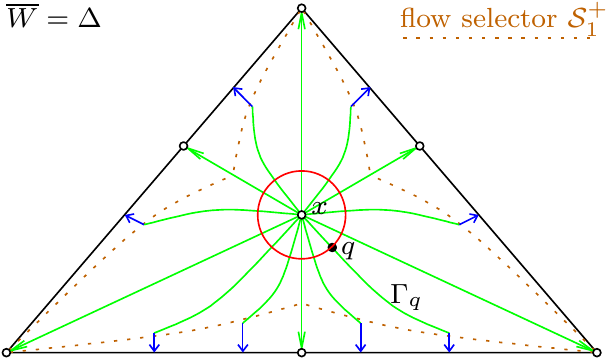}
  \caption{Curves $\Gamma_q$ composed of flows
                 {\color{green} $\varphi$} and {\color{blue} $\theta^1$}
                 partition $W=W^u(x)$
          }
  \label{fig:fig-filling-curves-a}
\end{figure}
  \\
  Now extend each dynamical thickening $(N_1,\theta^1)$
  of an index 1 point down to level $c_1-\eps_1$ via $\varphi$
  and then move the fibers that lie on level $c_1-\eps_1$ further down
  all the way via the level preserving diffeomorphisms $\hat\varphi_t$
  generated by the vector field $X=-\nabla f/\norm{\nabla f}^2$
  on $M\setminus \Crit f$.\footnote{
    Using $\hat\varphi$ ensures that fibers of different thickenings meet compatibly:
    Descending fibers will lie in level sets. So if one of them intersects a
    lower lying entrance set $N^+$ (contained in a level set itself),
    it (locally) lies
    completely in $N^+$. So on overlaps the $\theta$'s are transverse.
    }
  We get a fiber bundle $\Uu_1\to W$ that contains $N_1$ and carries a
  fiberwise forward flow defined via conjugation by the
  $\hat\varphi_t$ and still denoted by $\theta^1$.
  \\
  To make the dynamical thickening $(\Uu_1,\theta^1)$
  forward $\varphi$-attractive one constructs a flow selector
  $\Ss_1^+\subset\Uu_1$ and throws away from each fiber
  of $\Uu_1$ the part outside $\Ss_1^+$:
  Distribute the entrance hypersurface $N_1^+$,
  see~(\ref{eq:N^+}) and Figure~\ref{fig:fig-Brexit},
  utilizing a monotone smooth function,\footnote{
    say $\chi:(\tau,\infty)\to(0,\infty)$
    with $\chi^\prime<0$ and $\chi(t)\to\infty$ and $0$,
    as $t\to\tau$ and $\infty$, respectively.
    }
  similar in spirit to~\cite{Weber:2014e},
  to obtain a flow selector $\Ss_1^+:=\Phi N_1^+\subset\Uu_1$
  that bounds a $\varphi$-attractive fiber bundle $\Ss_1:=\Phi N_1\to W$.
  (A fiber of $\Ss_1\subset\Uu_1$ arises from a fiber of $\Uu_1$
  by throwing away the part not enclosed by $\Ss_1^+$.
  The fibers of $\Ss_1$ are invariant under $\theta^1$.)
  This defines the $\varphi$-attractive dynamical thickening
  $(\Ss_1:=\Phi N_1,{\color{blue}\theta^1})$.
  The dotted line in Figure~\ref{fig:fig-filling-curves-a}
  shows the flow selector $\Ss_1^+$.
  The curves $\Gamma_q$ are composed of $\color{green}\varphi$ trajectories
  followed by $\color{blue}\theta^1$ trajectories -- transition taking place
  along the flow selector $\Ss_1^+$.
  The curves $\Gamma_q$ partition the unstable manifold
  $W$ of $x$ and its endpoints cover the boundary $\Delta$.
  The endpoints of the $\Gamma_q$ vary continuously in the elements
  $q$ of the descending sphere $\color{red} S^u_\eps$ as both
  $\color{green}\varphi$ and $\color{blue}\theta^1$ are transverse
  to~the~hypersurface~$\Ss_1^+$.
\end{itemize}

\vspace{.2cm}
\noindent
\textbf{Organization of this paper.}
In Section~\ref{sec:N-L} we prove the defining properties (i--iv) for Conley pairs
$(N,L)$ associated to \emph{isolated} critical points
and construct the open contractible thickenings of the stable manifolds
yet missing in the proof of the Lusternik-Schnirelmann
Theorem~\ref{thm:LS}.
Section~\ref{sec:LS} reviews
further tools to detect critical points:
Cuplength in cohomology, its dual cousin the subordination number,
and a variant $\catamb(M)$ of the LS category 
called ambient LS category.
For further reading, concerning LS theory, we recommend the
comprehensive monograph~\cite{cornea:2003a}
or the more elementary concise presentation in~\cite[IV.3]{zehnder:2010a}.

\vspace{.2cm}
\noindent
\textit{\small Acknowledgements.}
{\small
The author would like to thank brazilian tax payers
for excellent teaching and research conditions at UNICAMP.
}

\vspace{.2cm}
\noindent
For convenience of the reader we conclude the introduction
by introducing these tools and summarize their interactions.
Throughout $R$ is a commutative ring:

\subsubsection*{Cuplength}
The \textbf{$\mbf{R}$-cuplength} $\cupp_R(X)$ of a
topological space $X$ is the largest integer $k\in\N$ such that there
exist $k$ cohomology classes $\mbf{\alpha_1},\dots,\mbf{\alpha_{k}}$
of \emph{positive} degree (grading) in the cohomology ring
$\Ho^*=\Ho^*(X;R)$ whose cup product is non-zero
\begin{equation}\label{eq:cupp}
     \mbf{\alpha_1}\CUP\dots\CUP\mbf{\alpha_{k}}\not= 0.
\end{equation}
If no such classes exist (cohomology in positive degree is $0$),
set $\cupp_R(X)=0$.
Recall that degrees add up under $\cup:\Ho^k\times\Ho^\ell\to\Ho^{k+\ell}$
and $\Ho^{k>\dim M}(M)=0$ for manifolds.
So the positive degree assumption implies the finiteness estimate
\begin{equation}\label{eq:cupp-dim}
     \cupp_R(M)\le\dim M
\end{equation}
for any \emph{connected}\footnote{
  Connectedness is crucial, as the RHS is
  independent of the component~number.
  }
manifold $M$. Many $R$-cuplengths of many common manifolds are
known; see e.g.~\cite[\S 1.2]{cornea:2003a}. This, together with the
key estimate $\abs{\Crit f}>\cupp_R(M)$, see~(\ref{eq:inequalities})
below, makes $\cupp_R$ a rather useful quantity.

\subsubsection*{Subordination}
For non-trivial homology
classes $\mbf{b_1},\mbf{b_2}\in\Ho_*(M;R)\setminus\{0\}$
of  a closed manifold $M$
one writes $\mbf{b_1}<\mbf{b_2}$ and says
$\mbf{b_1}$ \textbf{is subordinated to} $\mbf{b_2}$,
if there is a cohomology class $\mbf{\omega}\in\Ho^{p>0}$ of
\emph{positive} degree such that
$$
     \mbf{b_1}=\mbf{\omega}\CAP\mbf{b_2}
$$
where $\CAP:\Ho^p\times\Ho_m\to\Ho_{m-p}$ is the cap product.
Subordination is transitive and the degree strictly increases.
The \textbf{$\mbf{R}$-subordination number}
$\sub_R(M)$ is the largest integer $k\in\N$
such that there is a chain of $k$ subordinated classes
$$
     \mbf{b_1}<\mbf{b_2}<\dots<\mbf{b_k}<\mbf{b_{k+1}}.
$$
The significance of subordination lies in the fact that existence
of classes $\mbf{b_1}<\mbf{b_2}$ guarantees existence of two
different critical values, thus critical points, for any $C^2$ function $f$
whose critical points are all isolated; see Theorem~\ref{NEW-thm:LS-princple}
(Lusternik-Schnirelmann refined minimax principle).

\subsubsection*{Inequalities and comparisons}
For a closed manifold $M$ equipped with
a $C^2$ function $f$ there are the inequalities
\begin{equation}\label{eq:inequalities}
     \Abs{\Crit f}
     \ge
     \catamb(M)
     \ge
     \cat(M)
     >
     \cupp_R(M)
     =
     \sub_R(M)
\end{equation}
where $R$ is any commutative ring.
In the non-degenerate (Morse) case it holds
\begin{equation}\label{eq:Morse-subord-compar}
     \Abs{\Crit f}
     \ge\dim \Ho_*(M;\F)
     \ge 1+\sub_\F(M).
\end{equation}
for any field $\F$. All these inequalities will be proved below.

For any given field $\F$ Morse theory gives
by~(\ref{eq:Morse-subord-compar})
stronger (or equal)\footnote{
  Example for 'equal':
  $\dim\Ho_*(\RP^2;\Z_2)=3=1+\cupp_{\Z_2}(\RP^2)$ or for $\F=\Q$: $1=1+0$.
  }
critical point estimates than subordination/cuplength.
In contrast, category can be superior
to Morse theory, depending on the field. Indeed
\begin{equation}\label{eq:exception}
     \dim\Ho_*(\RP^2;\Q)=1+0+0<3=\cat(\RP^2).
\end{equation}
But in case of $\RP^2$ there \emph{still exists}
a field bringing back Morse theory, namely
$$
     \dim\Ho_*(\RP^2;\Z_2)=1+1+1=\cat(\RP^2).
$$
Is this a general fact? For simply connected orientable manifolds the
answer is yes: These satisfy $\dim\Ho_*(M;\Q)\ge\cat(M)$
by~\cite[Ex.\,1.33, cf. p.\,291]{cornea:2003a}.

\section{Conley pairs for isolated critical points}\label{sec:N-L}
\begin{definition}\label{def:index-pair}
Let $\phi=\{\phi_t\}$ be a continuous flow on a topological space~$X$.
A~{\bf\boldmath Conley pair $(N,L)$
for an isolated fixed point $x$} of $\phi$
consists of a pair of compact subspaces $(N,L)$
of $X$ which satisfy
\begin{enumerate}
  \item[(i)]
     $x\in \interior{N}\setminus L$
  \item[(ii)]
     $N\cap\Fix\,\phi=\{x\}$
  \item[(iii)]
     $p\in L$ and $\phi_{[0,t]} p\subset N
     \,\Rightarrow\,
     \phi_t p\in L$
  \item[(iv)]
     $p\in N$ and 
     $\phi_T p\notin N
     \;\Rightarrow\;
     \exists \sigma\in[0,T):
     \phi_\sigma p\in L$ and
     $\phi_{[0,\sigma]} p\subset N$
\end{enumerate}
\end{definition}

In particular, conditions~(i) and~(ii) tell that $N$ is a neighborhood
of $x$ which contains no other critical points in its closure.
Condition~(iii) says that {\bf\boldmath$L$ is positively
invariant in $N$} and~(iv) asserts that every semi-flow line
which leaves $N$ goes through $L$ \emph{before} exiting.
Hence we say that {\bf\boldmath $L$ is an exit set of~$N$};
cf. Figures~\ref{fig:fig-isolated} and~\ref{fig:fig-non-degenerate}.
The set $N$ is also called a \textbf{Conley block}.
Note that in this generality, as opposed to the
realization~(\ref{eq:N-iso}) of $(N,L)$ for downward gradient flows,
there is no obstruction that exiting points would re-enter $N$.
For downward gradient flows the assumption in~(iii)
is equivalent to~(\ref{eq:ass-iii-altern}).

\vspace{.2cm}
Coming back to gradient flows suppose from now on, throughout
Section~\ref{sec:N-L}, that $f:M\to\R$ is a $C^2$ function
on a closed Riemannian manifold $(M,g)$.

\vspace{.1cm}
\noindent
\textit{Preparing the next two proofs.}
Pick two regular values $a<b$ of $f$ such that
there is only one critical value $c$ in between them
which, moreover, is their mean $c=\frac{a+b}{2}$.
By  $f:f^{-1}[a,b]\to[a,b]$ we denote the restriction to the
(compact) domain $f^{-1}[a,b]$.
Let $\varphi$ be the corresponding (local downward) gradient flow.
Furthermore, suppose that $x$ is an \emph{isolated} critical point of $f$.
Let $(N,L)$ be defined by~(\ref{eq:N-iso}) and~(\ref{eq:L-iso})
with constants $\eps\in(0,\frac{b-a}{2}]$ and $\tau\ge 1$.

\begin{lemma}[Shrink to critical point]\label{le:shrink-to-pt}
Let $U_x$ be a neighborhood in $f^{-1}[a,b]$
of the isolated critical point $x$ of $f$.
Then there are constants $\eps_*>0$ and $\tau_*\ge 1$ such that
$N_x^{\eps_*,\tau_*}$ is contained in $U_x$.
\end{lemma}

\begin{proof}
Write the set of critical points of $f:f^{-1}[a,b]\to[a,b]$
as disjoint union $\{x\}\mathop{\dot{\cup}} C$
of two (compact) subsets.
Pick disjoint open neighborhoods $U$ of $x$ and $V$ of $C$.
Suppose that $U\subset U_x$; otherwise, replace $U$ by $U\cap U_x$.
Observe that the complement $K$ of $U\cup V$ is compact and
contains no critical points.

Now suppose by contradiction that the set $N_x^{\eps,\tau}$ was
not contained in $U$ for all $\eps,\tau$. Then there are sequences
$\eps_\nu\searrow 0$ and $\tau_\nu\nearrow\infty$ such that
$N_\nu:=N_x^{\eps_\nu,\tau_\nu}\not\subseteq U$, that is
$N_\nu\setminus U\not= \emptyset$, for every $\nu\in\N$.
More is true, namely\footnote{
  Otherwise $N_\nu$ must contain at least one element of $V$
  and there would be the inclusion
  $N_\nu\subset U \mathop{\dot{\cup}} V$.
  The latter provides the first of the two identities
  $N_\nu=N_\nu\cap(U \mathop{\dot{\cup}}
  V)=(N_\nu\cap U) \mathop{\dot{\cup}} (N_\nu\cap V)$.
  As also $N_\nu\cap U\ni x$ is non-empty,
  this contradicts connectedness of $N_\nu$.
  }
$$
     N_\nu\setminus (U\cup V)\not= \emptyset.
$$
Thus there is a sequence
$p_\nu\in N_\nu\setminus (U\cup V)\subset K$
and a subsequence, still denoted by $p_\nu$,
that converges to some point $p\in K$, as $\nu\to\infty$;
see Figure~\ref{fig:fig-shrinking-lemma}.
\begin{figure}
  \centering
  \includegraphics
                             {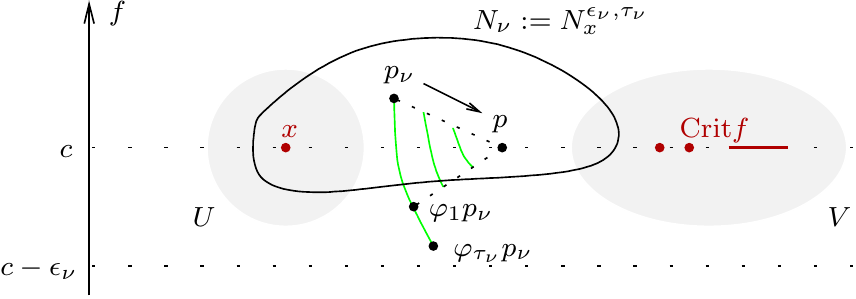}
  \caption{Proof of the shrinking Lemma~\ref{le:shrink-to-pt}}
  \label{fig:fig-shrinking-lemma}
\end{figure}
The fact that $p_\nu\in N_\nu$ implies firstly that
$f(p)\le c$, since $f(p_\nu)\le c+\eps_\nu$,
and secondly that $f(\varphi_1 p)\ge c$, since $f(\varphi_1 p_\nu)
\ge f(\varphi_{\tau_\nu} p_\nu)\ge c-\eps_\nu$.
Since a downward gradient flow strictly decreases $f$, except at critical
points, we get for $t\in[0,1]$ that
$$
     c\ge f(p)\ge f(\varphi_t p)\ge  f(\varphi_1 p)\ge c,
$$
i.e. $f(\varphi_t p)\equiv c$, $\forall t\in[0,1]$.
So $p\in K$ is a critical point of $f$. Contradiction.
\end{proof}

\begin{theorem}[Conley pair]\label{thm:Conley-pair}
The pair $(N,L)$ defined by~(\ref{eq:N-iso}--\ref{eq:L-iso})
is a Conley pair for an isolated
fixed point $x$ of $\varphi$ for all $\eps>0$ small and $\tau\ge 1$ large.
\end{theorem}

\begin{remark}[No re-entry]\label{rem:no-re-entry}
Since we work with a downward \emph{gradient}
flow $\varphi$ the function $f$ decays along trajectories.
Now observe that a point which leaves $N$ will precisely $\tau$ time units later
run through the level $c-\eps$. But $N$ itself sits strictly above that level.
Therefore a point which exits $N$ cannot re-enter.
\end{remark}

\begin{proof}[Proof of Theorem~\ref{thm:Conley-pair}]
We need to verify properties~(i--iv) in Definition~\ref{def:index-pair}.

(i)~Because $f(x)=c$ and the critical point $x$ is a fixed point of
$\varphi$ it is clear that $x\in N$ by definition~(\ref{eq:N-iso}).
Since $f$ and $f\circ \varphi_\tau$ are continuous $x$ lies in the
interior of $N$. One has $x\notin L$, because $f(\varphi_{2\tau} x)=f(x)=c$.

(ii)~True by the shrinking Lemma~\ref{le:shrink-to-pt}. Here
isolatedness of $x$ enters.

(iii)~As $f$ decreases along $\varphi$,
the assumption in~(iii) is equivalent to
\begin{equation}\label{eq:ass-iii-altern}
     p\in L\quad\wedge\quad \varphi_t p\in N,
\end{equation}
for some $t\ge 0$. This implies that $\varphi_tp\in L$:
Indeed $\varphi_tp\in N$ by assumption and
$$
     f(\varphi_{2\tau}\varphi_t p)\le f(\varphi_{2\tau} p)\le c-\eps.
$$
Step one uses $t\ge 0$ and that $f$ decreases along $\varphi$.
Step two holds since $p\in L$.

(iv)~Suppose $p\in N$ and $\varphi_T p\notin N$ for some $T>0$.
There are two cases.
\noindent
Case 1: $p\in L$. Pick $\sigma=0$.

\noindent
Case 2: $p\in N\setminus L$. Note that $N\setminus L$ is open in $N$.
We need to find a time $\sigma\in(0,T)$ such that
$\varphi_\sigma p\in L$ and $\varphi_{[0,\sigma]} p\subset N$.
By~(ii) the only critical point in $N$ is $x$.
The assumptions imply firstly that $p$ is not a critical point,
but is connected to $x$ inside $N$ through a continuous path,
and secondly that
\begin{equation*}
     f(p)\le c+\eps,\qquad
     f(\varphi_{2\tau} p)>c-\eps,\qquad
     f(\varphi_{\tau+T} p)<c-\eps.
\end{equation*}
We will show that these \textbf{three inequalities} imply that,
firstly, there is a unique time $\alpha> 0$ until which the trajectory
through $p\in N\setminus L$ remains in $N$ and at which it
enters $L$ and, secondly, that there is a unique time
$\beta\in(\alpha,T)$ at which the trajectory leaves $L$, hence
by~(iii) simultaneously $N$, forever (Remark~\ref{rem:no-re-entry}).
Given $\alpha$ and $\beta$, any $\sigma\in[\alpha,\beta]
\subset(0,T)$ satisfies the conclusion of~(iv).

To define the entrance time $\alpha>0$ observe that by inequalities
two and three, 
together with the fact that $f$ decays along $\varphi$, the trajectory through
$p\in N\setminus L$ runs through the level set $\{f=c-\eps\}$ at a
unique time $\mbf{T_*}\in(2\tau,\tau+T)$. Set
$$
     \alpha:=T_*-2\tau>0
$$
to obtain
$
          c-\eps
          =f(\varphi_{T_*}p)
          =f(\varphi_{2\tau+\alpha}p)
$.
To get $T_*=2\tau+\alpha=\tau+\beta$ define
$$
     \beta:=\alpha+\tau\ge\alpha+1.
$$
So the identity reads
$
          c-\eps
          =f(\varphi_{\tau+\beta}p)
$.
Thus $\beta<T$ by inequality three.

It remains to show, firstly, that $\alpha>0$ is the unique time
at which the trajectory through $p\in N\setminus L$ enters $L$ and at
least until which it lies in $N$ and, secondly, that $\beta$ is the
unique time when the trajectory leaves $L$, thus $N$.
More precisely, we show, firstly, that $\varphi_s p\in N$ for some $s\ge 0$
if and only if $s\in[0,\beta]$ and, secondly, that $\varphi_s p\in L$
for some $s>0$ if and only if $s\in[\alpha,\beta]$.
\newline
\textsc{Assertion 1.}
Pick $s\in[0,\beta]$. Then
$
     f(\varphi_s p)
     \le f(p)
     \le c+\eps
$
since $p\in N$. Furthermore, note that
$\tau+s\le\tau+\beta=T_*$.
So
$
     f(\varphi_\tau(\varphi_s p))
     \ge f(\varphi_{T_*} p)
     =c-\eps.
$
Moreover, via $[0,s]\ni t\mapsto \varphi_t p$, the point
$\varphi_s p$ path-connects inside $N$ to $p$ which in turn path-connects
inside $N$ to $x$ by definition of $p\in N$.
This proves that $\varphi_s p\in N$.
Vice versa, assume $\varphi_s p\in N$ for some $s\ge 0$.
The desired inequality $s\le \beta$ is equivalent to $s+\tau\le T_*$.
As $f(\varphi_{T_*} p)=c-\eps$, the latter inequality follows from
the consequence $f(\varphi_{\tau+s} p)\ge c-\eps$ of the assumption
$\varphi_s p\in N$ and the gradient flow property that $f$ decreases
along the trajectory.
\newline
\textsc{Assertion 2.}
Pick $s\in[\alpha,\beta]$. Then $\varphi_sp\in N$
by assertion 1. It remains to show
$
     f(\varphi_{2\tau}(\varphi_sp))
     \le c-\eps
$.
This holds true since $c-\eps=f(\varphi_{T_*} p)$
and $2\tau+s\ge 2\tau+\alpha=T_*$
by choice of $s$ and definition of $\alpha$.
Vice versa, assume $\varphi_sp\in L$
for some $s>0$. Then we get the two inequalities
$
     f(\varphi_\tau(\varphi_sp))
     \ge c-\eps
$
and
$
     f(\varphi_{2\tau}(\varphi_sp))
     \le c-\eps
$
by definition of $L$. If $s>\beta$, equivalently
$s+\tau>\beta+\tau=T_*$, we get
$
     f(\varphi_{s+\tau}p)
     < f(\varphi_{T_*}p)
     =c-\eps
$ contradicting inequality one.
In the case $s\in(0,\alpha)$ we get
$
     f(\varphi_{2\tau+s}p)
     >f(\varphi_{T_*}p)
     =c-\eps
$ contradicting inequality two.
\end{proof}

\subsubsection*{Thickenings of (un)stable manifolds via Conley pairs}
\begin{proposition}\label{prop:amb-thick}
Suppose $f$ is a $C^2$ function on a closed Riemannian manifold $(M,g)$
with isolated, thus finitely many, critical points,
say $x_1,\dots,x_\ell$. Then
\begin{itemize}
\item[\rm (i)]
  there is an open cover
  $\{\Ww_i\}_{i=1}^\ell$ of $M$ by nullhomotopic
  \textbf{thickenings of the unstable manifolds},
  that is each $\Ww_i$ is open in $M$,
  contains the unstable manifold $W_i$ of $x_i$, and is
  nullhomotopic to $x_i$;
\item[\rm (ii)]
  there is an open cover $\{\Uu_i\}_{i=1}^\ell$ of $M$
  where each $\Uu_i$ is ambient\,\footnote{
    see Definition~\ref{def:amb-LS-cat}
    }
  nullhomotopic to $x_i$ (and covers 'large parts' of the unstable manifold).
\end{itemize}
Thickenings corresponding to critical points on the same level set are
pairwise disjoint. Furthermore, there are analogous open covers
$\{\Ww_i^*\}_{i=1}^\ell$ and $\{\Uu_i^*\}_{i=1}^\ell$ corresponding to
the stable manifolds.
\end{proposition}

\begin{proof}
The proof takes three steps (0), (i), (ii).
Step (0) is taken from~\cite{Robbin:2002}.

\textbf{(0)}~\textit{Each critical point $x_i$ has an ambient nullhomotopic
open neighborhood $V_i$}:
As the critical points $x_i$ are isolated,
there are pairwise disjoint local coordinate charts
$\{(\psi_i,U_i)\}_{i=1}^\ell$ such that $U_i$ contains no critical
point except $x_i$.
Pick an open ball about $\tilde x_i:=\psi_i(x_i)$ contained in
$\tilde U_i:=\psi_i(U_i)$ and another open ball $\tilde V_i$ about
$\tilde x_i$ of, say, half the radius of ball one.
There is a simple radial homotopy of smooth
maps that deforms the closure of the smaller ball to its center $\tilde x_i$ while
the points in the complement of the larger ball remain fixed; just
stretch the annulus. Set $V_i:={\psi_i}^{-1}\tilde V_i$.
Note that by construction the closures ${\widebar{V}}_{{\mspace{-5mu} i}}$ are
pairwise disjoint and also ambient nullhomotopic.\footnote{
  At this point one might be tempted to define the thickening $\Ww_i$
  as the set that is exhausted by $V_i$ in forward time, that is
  $\varphi_{[0,\infty)} V_i$, then homotop that set back
  into ${\widebar{V}}_{{\mspace{-5mu} i}}$
  followed by the contraction to $x_i$ from Step~(0).
  But how to \emph{continuously}
  deform $\varphi_{[0,\infty)} V_i$ back into the closure of $V_i$?
  The obvious deformation of just following the backward flow lines
  until meeting ${\widebar{V}}_{{\mspace{-5mu} i}}$
  may lack continuity due to the possibility that some flow lines may
  leave 
  and re-enter again.
  }

\textbf{(i)} A way to control the problem
of multiple entrance and exit times is to use a Conley pair
$(N_i,L_i)$ for $x_i$ as provided by Theorem~\ref{thm:Conley-pair}.
\begin{figure}
  \centering
  \includegraphics
                             {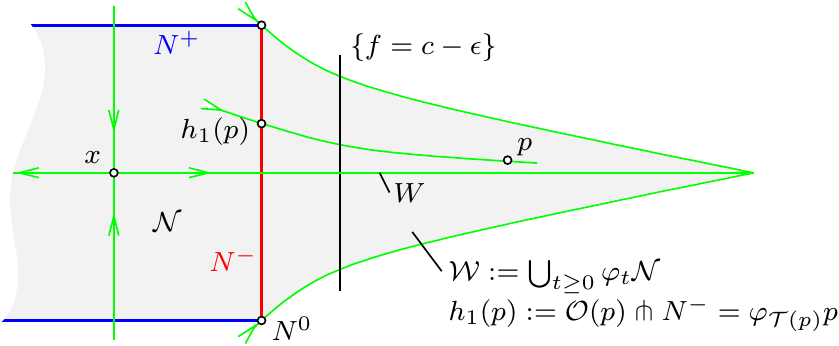}
  \caption{Thickening $\Ww$ of $W=W^u(x)$ as Conley block forward exhaustion}
  \label{fig:fig-Brexit}
\end{figure}
By isolatedness of $x_i$ the shrinking Lemma~\ref{le:shrink-to-pt}
applies and we may suppose that $N_i\subset V_i$.
Since the $V_i$ are pairwise disjoint, so are the $N_i$.
Actually we need here only the 'outmost' points of the exit set $L_i$,
namely, the so-called \textbf{exit locus}
$$
     {\color{red} N_i^-}
     :=\left(\{f<c_i+\eps_i\}\cap{\varphi_{\tau_i}}^{-1}\{f=c_i-\eps_i\}\right)_{N_i}
     ,\quad c_i:=f(x_i),
$$
that consists of those points of $N_i=N^{\eps_i,\tau_i}_{x_i}$
which lie below the upper level set and reach the lower one
precisely in time $\tau_i$.
Here $(\dots) _{N_i}$ selects those connected
components that lie in $N_i$.
Similarly there is the \textbf{entrance locus}
\begin{equation}\label{eq:N^+}
     {\color{blue} N_i^+}
     :=\left\{ p\in\{f=c_i+\eps_i\} \mid
     f(\varphi_{\tau_i}p)>c_i-\eps_i\right\}_{N_i}
\end{equation}
and the \textbf{bounce off locus}
\begin{equation*}
\begin{split}
     N_i^0
   &:=\left\{ p\in N_i\cap\{f=c_i+\eps_i\} \mid
     f(\varphi_{\tau_i}p)=c_i-\eps_i\right\}\\
   &\,\,=\left(\{f=c_i+\eps_i\}\pitchfork
     {\varphi_{\tau_i}}^{-1}\{f=c_i-\eps_i\}\right)_{N_i}.
\end{split}
\end{equation*}
Note that $N_i^-$ and $N_i^+$ are open subsets of
the hypersurfaces ${\varphi_{\tau_i}}^{-1}\{f=c_i-\eps_i\}$ and
$\{f=c_i+\eps_i\}$, respectively, whereas $N_i^0$ 
consists of components of their transverse intersection.
So the $N_i^\pm$ are non-compact hypersurfaces of $M$
and $N_i^0$ is a closed codimension 2 submanifold. Let $\Nn_i$ be the
interior and $\boundary{N}_i$ the topological boundary of the Conley
block $N_i$. Figure~\ref{fig:fig-Brexit} illustrates the partitions
\begin{equation}\label{eq:N-partition}
     \boundary{N}_i={\color{blue} N_i^+} \mathop{\dot{\CUP}} N_i^0
     \mathop{\dot{\CUP}}{\color{red} N_i^-} ,
     \qquad
     N_i=\Nn_i\mathop{\dot{\CUP}}\boundary{N}_i ,
     \qquad
     \Nn_i:=\interior{N}_i.
\end{equation}

Define the desired thickening to be the forward exhaustion of the
interior $\Nn_i$ of the Conley block $N_i$, namely
$$
     \Ww_i:=\varphi_{[0,\infty)} \Nn_i
     :=\bigcup_{t\ge 0} \varphi_t \Nn_i.
$$
The set $\Ww_i$ is open in $M$ and contains the whole unstable manifold
$W_i$ along which it extends. A homotopy $h_\lambda:\Ww_i\to M$,
$\lambda\in[0,1]$, between the inclusion $h_0:\Ww_i\hookrightarrow M$
and a map $h_1:\Ww_i\to M$ whose image lies in $N_i$ is given by
\begin{equation*}
\begin{split}
     h:\I\times\Ww_i\to M,\quad(\lambda,p)\mapsto
     \begin{cases}
        p&\text{, $p\in\Ww_i\cap N_i$,}
        \\
        \varphi_{\lambda\Tt(p)} p&\text{, $p\in\Ww_i\setminus N_i$.}
     \end{cases}
\end{split}
\end{equation*}
Here $\Tt(p)<0$ is the arrival time of $p\in\Ww_i\setminus N_i$
at the exit locus $N_i^-$. It is well defined since such $p$ comes
from the interior of $N_i$ by definition of $\Ww_i$, so it must have
left through the exit set $L_i$, thus through $N_i^-$, by
property~(iv) in Definition~\ref{def:index-pair}.
This is illustrated by Figure~\ref{fig:fig-Brexit}
in terms of the orbit $\Oo(p)$ through $p$.
That orbit is orthogonal to the lower level set $\{f=c_i-\eps_i\}$,
hence still transverse to the time $-\tau_i$ copy $N_i^-$.
But this means that the arrival time $\Tt(p)$
varies continuously in $p$.
The piecewise definition of $h$ also matches continuously:
For a point $p\in\Ww_i\setminus N_i$ close to the
other set $\Ww_i\cap N_i=\Nn_i\cup N_i^-$,
hence close to $N_i^-$, the arrival time approaches 0,
so $\varphi_{\Tt(p)} p$ approaches $p$.

The desired contraction is then given by the homotopy $h$ from
$\Ww_i$ to $N_i\subset V_i$ followed by the ambient nullhomotopy
in Step~(0) of $V_i$ onto the critical point $x_i$.
The collection $\Ww_1,\dots,\Ww_\ell$ covers $M$
since already the unstable manifolds do.
Those $\Ww_i$ corresponding to critical points on the same
level, say $c$, are pairwise disjoint: Indeed the $V_i$, thus the $\Nn_i$, are
and every point $p$ of $\Ww_i$ outside $\Nn_i$
has left through the exit locus $N_i^-\subset L_i$.
Thus $p$ has crossed or will cross level $c-\eps_i$
by definition of $L_i$. But such flow line cannot enter any
of the other $\Ww_j$'s since their entrance loci $N_j^+$
lie on the higher level $c+\eps_j$.

\begin{figure}
  \centering
  \includegraphics
                             {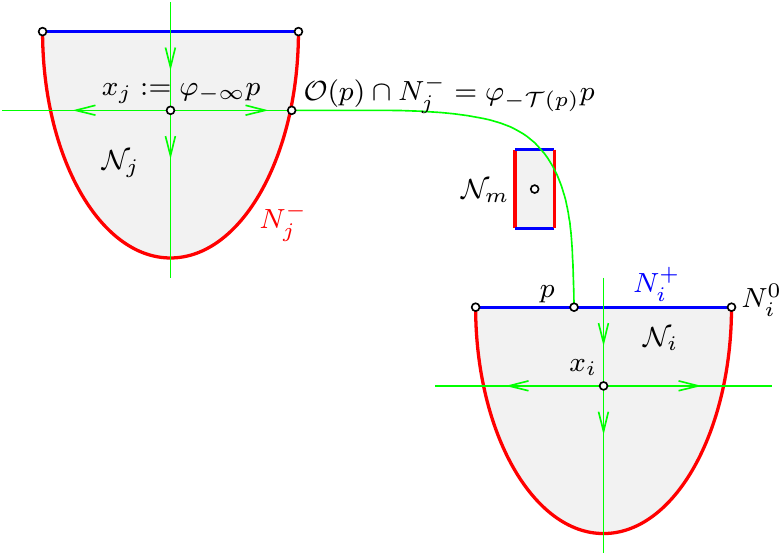}
  \caption{Backward entrance time $\Tt_i^+:N_i^+\to\R_+$
                 and $N_{i}$ for $f=-u_1^2+u_2^3$}
  \label{fig:fig-T_p}
\end{figure}

\textbf{(ii)}~After handy first tries\footnote{
  \emph{Infinite exhaustion.}
  To construct \emph{ambient nullhomotopic} open sets $\Uu_i$ that
  cover $M$ one feels that the required map $h_1:M\to M$ homotopic
  to the identity, see Definition~\ref{def:amb-LS-cat}, should come
  from the flow $\{\varphi_t\}$ provided by the problem. It is tempting to
  try the \emph{infinite} time exhaustion
  $\varphi_{[0,\infty)} \Nn_i$ from~(i).
  Unfortunately, at infinity there are fixed points of $\varphi$ which cause
  cracking/discontinuity of natural (flow induced) \emph{ambient} homotopies.

  \emph{Finite exhaustion.}
  So let's try $\varphi_{[0,T]} \Nn_i$ for some \emph{finite}
  time $T\ge 0$. For large $T$ this set covers a major part of the unstable
  manifold. Good. But applying the backward flow $\varphi_{-T}:M\to M$
  does not, in general, move the set back to $\Nn_i$!
      Indeed as $\varphi_{[0,T]} \Nn_i$ contains $\Nn_i$
      the pre-image $\varphi_{-T} \varphi_{[0,T]} \Nn_i$
      contains $\varphi_{-T} \Nn_i$ -- a set that
      crawls up the stable manifold of $x_i$.

  \emph{Time-$T$ image.}
  The problem disappears if one tries as a candidate for
  $\Uu_i$ the image $\varphi_T \Nn_i$
  under just one time-$t$-map. This open set moves back correctly to $\Nn_i$
  under the backward flow $\varphi_{-t}$ when $t$ runs from $0$ to $T$.
  This set almost covers the unstable manifold for large
  $T$. But it not only stretches out along $W_i$, as $T$ grows,
  but also gets 'thinner'. Is there a uniform~$T$?
  }
we start from scratch and, for a change, emphasize stable
manifolds and backward flow in order to construct ambient
contractible sets $\Uu_i^*$ that crawl up the stable manifolds and are
of the form $\varphi_{-T_i}\Nn_i$. (Replace $f$ by $-f$ to get
$\Uu_i$.)
For each of the $\ell$ critical points $x_i$ of $f$ pick a Conley pair
$(N_i,L_i)$ as in Theorem~\ref{thm:Conley-pair}.
By finiteness of $\Crit f$ suppose that all of these pairs
are defined with respect to the same constants $\eps$ and $\tau$
chosen, in addition, such that $N_i\subset V_i$ by the Shrinking
Lemma~\ref{le:shrink-to-pt}. By Step~(0) the $N_i$ are pairwise
disjoint. As earlier $\Nn_i$ denotes the interior of $N_i$.
\\
For each critical point $x_i$ set $c_i:=f(x_i)$ and
consider the function
$$
     \Tt_i^+:N_i^+\to (0,\infty),
$$
as illustrated by Figure~\ref{fig:fig-T_p},
that maps a point $p$ of the entrance locus
${\color{blue} N_i^+}\subset\{f=c_i+\eps\}$
to the time $\Tt_i^+(p)$ at which the backward flow of $p$
meets the exit locus ${\color{red} N_j^-}$ associated to $p$'s asymptotic
origin~$x_j:=\varphi_{-\infty}p$.

\vspace{.1cm}\noindent
{\it Remark.}
Note that $\varphi_{-\Tt_i^+(p)}p$ lies in the boundary of the descending
disk $W^u_\eps(x_j)$, hence $\varphi_{-\mu-\Tt_i^+(p)}p$
lies in its interior, hence in $\Nn_j$, whenever $\mu>0$.

\vspace{.1cm}\noindent
The number $\Tt_i^+(p)$ is well
defined and finite, because the asymptotic backward limit
$\varphi_{-\infty}p:=\lim_{t\to-\infty}\varphi_t p$ exists and is a
critical point, say $x_j$, which sits inside the Conley block $N_j$.
Since any two Conley blocks are disjoint the time $\Tt_i^+(p)>0$ is
positive. Although the function $\Tt_i$ might be highly
discontinuous (perturb $p$ in Figure~\ref{fig:fig-T_p} slightly to the right),
it is bounded: The closures of $N_i^+$ and the finitely
many higher lying $N_j^-$ are all compact and, most importantly,
they contain no critical points by Theorem~\ref{thm:Conley-pair}.
For each critical point $x_i$ set
$$
     \Tt_i:=1+\sup_{N_i^+}\Tt_i^+\in[0,\infty)
$$
where the convention $\sup_{\emptyset}\Tt_i^+:=0$ takes care of local maxima.

\vspace{.1cm}\noindent
{\it Not a good idea.}
One might try for the required open ambient nullhomotopic cover
of $M$ the sets $\Uu_i^{*\prime}:={\varphi_{\Tt_i}}^{-1}\Nn_i$
(or utilize some uniform time, say $\Tt_*:=\max_i\Tt_i$, and try
$\Uu_i^{*\prime\prime}:={\varphi_{T_*}}^{-1}\Nn_i$).
The proof below will work -- except for the final argument,
case 2 b). The problem will be that with this choice
one does get back from $N_i^+$ to $N_j^-$, but the relevant set
to backward-enter $\Uu_j$
is ${\varphi_{T_j}}^{-1} N_j^-$ which sits way further in the past.
This indicates that one should define the sets $\Uu_i$ successively
according to the level of $x_i$ starting with the highest one
and adding 'extra backward time' in the definition of the
lower level $\Uu_i$'s.

\vspace{.1cm}\noindent
{\it Definition of $\Uu_1^*,\dots,\Uu_\ell^*$.}
Suppose there are $\ell$ critical points of $f$ and these
are enumerated such that $c_1=f(x_1)\le\dots\le c_\ell=f(x_\ell)$.
Set $T_{\ell+1}=0$ and define
$$
     T_i:=\Tt_i+T_{i+1}=\Tt_i+\dots+\Tt_\ell ,\qquad
     \Uu_i^*:={\varphi_{T_i}}^{-1}\Nn_i ,
$$
for every $i=1,\dots,\ell$.

\vspace{.1cm}\noindent
{\it We finish by verifying the required properties for the collection
of sets $\Uu_1^*,\dots,\Uu_\ell^*$.}
Any set $\Uu_i^*$ is open as the interior $\Nn_i$ of $N_i$ is open
and the map $\varphi_{T_i}:M\to M$ is continuous.
The family $h_i:=\{\varphi_t\}_{t\in[0,T_i]}$ applied to
$\Uu^*_i$ is an ambient homotopy from $\Uu_i^*$ to
$\Nn_i\subset V_i$, then apply Step~(0) to $V_i$ to arrive at $x_i$.
That those sets $\Uu_i^*$ which correspond to critical
points on the same level set are pairwise disjoint
follows as in Step~(i).
It remains to show that the sets $\Uu^*_1,\dots,\Uu_\ell^*$
cover $M$. To see this pick a point $p\in M$.
By closedness of $M$ and \emph{isolatedness} of $\Crit f$
the backward and forward asymptotic limits $\varphi_{\mp\infty}p$
exist and are critical points of $f$, say $x_j$ and $x_i$,
respectively. So
$$
     p\in W^u(x_j)\cap W^s(x_i).
$$
There are two cases.

\vspace{.1cm}\noindent
\textbf{Case 1} ($p\in \Nn_i$): As $p$, so by forward flow invariance
$\varphi_{T_i}p$, lies in the interior of
$W^s_\eps=W^s(x_i)\cap N_i$, it holds $\varphi_{T_i} p\in \Nn_i$.
So $p={\varphi_{T_i}}^{-1}(\varphi_{T_i}p)\in\varphi_{-T_i}(\Nn_i)=\Uu^*_i$.

\vspace{.1cm}\noindent
\textbf{Case 2} ($p\notin \Nn_i$):
Let $t_p\ge 0$ be the time when $p$ arrives at the entrance locus
$N_i^+$. There are two cases.
\textbf{a)}~If $T_i>t_p$, then
$$
     p={\varphi_{T_i}}^{-1}(\varphi_{T_i}p)\in{\varphi_{T_i}}^{-1}\Nn_i=\Uu^*_i
$$
and we are done. To see that $\varphi_{T_i}p\in\Nn_i$ notice that
$\varphi_{t_p}p\in N_i^+\cap W^s(x_i)=\p W^s_\eps$.
Hence at the larger time $T_i>t_p$ the point $\varphi_{T_i}p$
has moved from the boundary to the interior of the ascending disk
($-\nabla f$ is inward pointing), thus into $\Nn_i$.
\textbf{b)}~If $t_p\ge T_i$, i.e.
$t_p=\delta+T_i=\delta+\Tt_i+\dots+\Tt_{j-1}+T_j$
with $\delta\ge0$, then
$$
     p
     =\varphi_{-t_p}\underbrace{\varphi_{t_p}p}_{=:P}
     ={\varphi_{T_j}}^{-1}
       \bigl[\varphi_{-(\underbrace{\delta+\Tt_i+\dots+\Tt_{j-1}}_{\ge1+\Tt_i^+(p)})}
       \underbrace{P}_{\in N_i^+}\bigr]
     \in{\varphi_{T_j}}^{-1}\Nn_j=\Uu_j^* .
$$
Here we used the \textit{Remark} above to conclude
that the point in brackets $[\dots]$ lies in $\Nn_j$.
This concludes the proof of Proposition~\ref{prop:amb-thick}.
\end{proof}


\section{Lusternik-Schnirelmann theory}\label{sec:LS}
In this section we review proofs of the
inequalities~(\ref{eq:inequalities})
relating various lower bounds for the number
of critical points of a $C^2$ function $f$ on a closed manifold~$M$.
As a rule of thumb, Morse theory gives the strongest
lower bound, the sum of Betti numbers.
Exceptions confirming the rule include $\RP^2$; see~(\ref{eq:exception}).
But Morse theory is stronger (or equal) for simply connected
orientable closed manifolds using rational or real homology coefficients,
as detailed after~(\ref{eq:exception}).

\subsection{Lusternik-Schnirelmann categories}\label{sec:LS-cat}

\begin{definition}\label{def:LS-cat}
The \textbf{Lusternik-Schnirelmann category}
of a non-trivial topological space $X\not=\emptyset$, denoted by
$\cat(X)$, is the least integer $\ell\in\N$
such that there is an cover $U_1,\dots,U_\ell$ of $X$ by $\ell$
open nullhomotopic subsets $U_i\subset X$.\footnote{
  Definition of $\cat$ differs by $1$
  in the literature, e.g. the one in~\cite{cornea:2003a}
  is one less than ours.
  }
Such cover is called a \textbf{categorical cover}.
If there is no such (finite) cover, set $\cat(X):=\infty$.
The empty set is of category zero: By definition $\cat(\emptyset):=0$.
\end{definition}

\begin{remark}[Open versus closed covers]\label{rem:op-cl-cat}
If in the definition of the category one uses \emph{closed},
as opposed to open, sets $U_i$ one obtains the
\textbf{closed category} of $X$.
For paracompact Banach manifolds, hence for finite dimensional manifolds,
both definitions are equivalent; see e.g.~\cite[Prop.~1.10 \& App.~A]{cornea:2003a}.
\end{remark}

Note that if $W_1$ and $W_2$ are two components of a manifold,
then $\cat(W_1 \CUP W_2)=\cat(W_1)+\cat(W_2)$.
For connected\footnote{
  Assuming connectedness is crucial: The RHS is independent
  of the component~number.
  }
topological manifolds $W$ there is the non-trivial finiteness
estimate~\cite[Thm.~1.7]{cornea:2003a}
\begin{equation}\label{eq:cat-dim}
     \cat(W)\le 1+\dim W.
\end{equation}
The inequality is strict for all $n$-spheres with $n\ge 2$.

\subsubsection*{Ambient Lusternik-Schnirelmann category of manifolds}

\begin{definition}\label{def:amb-LS-cat}
Define the \textbf{ambient Lusternik-Schnirelmann category}
$\catamb(M)$ of a manifold $M$ the same way
as $\cat(M)$, but with \emph{nullhomotopic} replaced
by \emph{ambient nullhomotopic}:
  A subset $A\subset M$ is called
  \textbf{ambient nullhomotopic}
  if there is a differentiable map $h_1:M\to M$
  homotopic through such to the identity $h_0=\id_M:M\to M$ such that
  $h_1(A)=m$ for some $m\in M$.
\end{definition}

Clearly
$
     \catamb(M)\ge\cat(M)
$
as ambient nullhomotopic implies nullhomotopic.

\begin{example}[Nullhomotopic, but not ambient nullhomotopic]
\label{ex:cont-not-amb}
The open subset $U=\SS^2\setminus\{N\}$ of $\SS^2$, given by
the $2$-sphere minus the north pole, is a nullhomotopic subset,
but it is not ambient nullhomotopic.
\end{example}

In view of Proposition~\ref{prop:amb-thick}~(ii)
the proof of Theorem~\ref{thm:LS} also establishes

\begin{theorem}\label{thm:LS-amb}
Suppose $f$ is a $C^2$ function on a closed manifold $M$, then
$$
     \Abs{\Crit f}\ge \catamb(M).
$$
\end{theorem}

\subsection*{Cuplength}
In order to warm up let us first consider the case of real coefficients.

\begin{theorem}\label{thm:cup-catamb-F}
There is the strict inequality
$
     \cupp_\R(W)<\catamb(W)
$
for every manifold $W$. The cuplength $\cupp$ is defined by~(\ref{eq:cupp}).
\end{theorem}

The following proof is based on the
\textbf{de\,Rham model} of cohomology
$\Ho^*(M;\R)$ with real coefficients
where the $k$-cochains are sums of differential forms $\omega$ of degree $k$
and exterior differentiation $d$ being the boundary operator.
Cocycles are represented by closed differential forms $\omega$,
that is $d\omega=0$, and coboundaries by exact forms, that is those
of the form $d\theta$ for some $\theta$.

\begin{proof}
We cite~\cite{Robbin:2002} almost literally,
given its remarkable efficiency:
``Assume that $M=U_1\CUP\cdots\CUP U_\ell$ where the $U_i$
are open and that $f_i:M\to M$ ($i=1,\dots,\ell$) is a smooth map,
homotopic to the identity, such that $f(U_i)$ is a point. We must show
that $\omega_1\wedge\dots\wedge\omega_\ell$ is exact whenever
$\omega_1,\dots,\omega_\ell$
are closed forms of positive degree. Since
$\omega_i$ has positive degree, $f_i^*\omega_i|U_i=0$.
Since the sets $U_i$ cover $M$ we have
$(f_1^*\omega_1)\wedge\dots\wedge(f_\ell^*\omega_\ell)=0$.
Since $f_i$ is homotopic to the identity, there are forms $\theta_i$
with $\omega_i=f_i^*\omega_i+d\theta_i$.
Hence $\omega_1\wedge\dots\wedge\omega_\ell$ is a sum of products
$\beta_1\wedge\dots\wedge\beta_\ell$ where each $\beta_i$
is either $f_i^*\omega_i$ or $d\theta_i$ and at least one $\beta_i$
has the latter form. Each such product is exact so
$\omega_1\wedge\dots\wedge\omega_\ell$ is exact as
claimed.''\footnote{
  As $f_i\sim \id$, the difference $\mbf{f_i^*}-\mbf{\id^*}$ is zero on
  cohomology by the \texttt{(Homotopy)} axiom: So evaluating the
  pull-back difference $f_i^*-\id^*$ on any closed form, say
  $\omega_i$, returns an exact form.
  }
\end{proof}

Combining Theorems~\ref{thm:LS-amb} and~\ref{thm:cup-catamb-F}
shows that the $\R$-cuplength is a strict lower bound for the number
of critical points of a $C^2$ function on a closed manifold.
Let us now turn to the general case of coefficients
in any commutative ring $R$. The following result completes the proof
of the inequalities in~(\ref{eq:inequalities}).

\begin{theorem}\label{thm:cup-cat}
Given a topological space $X$, there is the strict inequality
$
     \cupp_R(X)<\cat(X)
$
whenever $R$ is a commutative ring.
\end{theorem}

\begin{proof}
Denote cohomology with coefficients in $R$ by $\Ho^*$.
Suppose $U_1,\dots,U_\ell$ is a categorical cover of $X$
and $\mbf{\alpha_1},\dots,\mbf{\alpha_\ell}\in\Ho^{\ge 1}(X)$ are
$\ell=\cat(X)$ cohomology classes of positive degree.
For each $k=1,\dots,\ell$ consider the two inclusion maps
$i_k:U_k\hookrightarrow X$ and $j_k:(X,\emptyset)\to (X,U_k)$
and the associated exact cohomology sequence of the pair $(X,U_k)$, namely
\begin{equation*}
     \dots 
     \overset{\:\delta^*}{\longleftarrow}
     \Ho^*(U_k)\overset{\mbf{\: i_k^\#}}{\longleftarrow}
     \Ho^*(X)\overset{\mbf{\: j_k^\#}}{\longleftarrow}
     \Ho^*(X,U_k)\overset{\:\delta^*}{\longleftarrow}
     \dots\quad .
\end{equation*}
Observe that $\mbf{\alpha_k}$ lies in the kernel of the (degree preserving)
homomorphism $\mbf{i_k^\#}$, because $\mbf{\alpha_k}$
is of positive degree $d_k>0$ while the target cohomology lives in degree zero
since $U_k$ is nullhomotopic. Thus, by exactness,
the class $\mbf{\alpha_k}$ is of the form $\mbf{j_k^\#}\mbf{\beta_k}$
for some relative class $\mbf{\beta_k}\in\Ho^{d_k}(X,U_k)$.
For excisive couples in $X$ (here openess and the
cover property of the $U_i$ enters, cf.~\cite[III~8.1]{dold:1995a})
the cup product descends to relative cohomology,
cf.~\cite[VII~(8.3')]{dold:1995a}, and we get that
\begin{equation*}
\begin{split}
     \mbf{\alpha_1}\CUP\dots\CUP\mbf{\alpha_\ell}
   &=\mbf{j_1^\#}\mbf{\beta_1}\CUP\dots\CUP\mbf{j_\ell^\#}\mbf{\beta_\ell}
     \in\Ho^{d_1+\dots+d_\ell}(X,U_1\CUP\dots\CUP U_\ell)
     =0
\end{split}
\end{equation*}
Here the last identity uses that $U_1\CUP\dots\CUP U_\ell=X$
and $\Ho^*(X,X)=0$.
\end{proof}

\subsection{Birkhoff minimax principle}
The second of the two pillars of Morse theory is the cell attachment
theorem~\cite[Thm.\,3.1]{milnor:1963a}, the first one is

\begin{theorem}[Regular interval theorem{\index{theorem!regular interval --}}]
\label{thm:no-crit_diff}
Assume $M$ is a manifold and $f:M\to\R$ is of class $C^2$ and
the pre-image $f^{-1}[a,b]$ is compact
and contains no critical points of~$f$. Then
$M^b:=\{f\le b\}$ and $M^a$ are diffeomorphic
Furthermore, the sublevel set $M^a$
is a strong deformation retract of $M^b$.
\end{theorem}

\begin{corollary}[Existence of a critical point]
\label{cor:ex-crit-by-reg-interv}\index{critical point!existence}
If two sublevel sets $M^a\subset M^b$ of a $C^2$ function
$f:M\to\R$ are of different homotopy type
and $f^{-1}[a,b]\subset M$ is compact, then there exists an
intermediate critical level, hence a critical point.
\end{corollary}

The idea to prove the regular interval
Theorem~\ref{thm:no-crit_diff}, namely, to exploit the
absence of critical points to push things down
also immediately implies the following version of the famous
\textbf{\Index{Birkhoff minimax principle}}~\cite{Birkhoff:1935b}.

\begin{theorem}[Minimax principle]\index{minimax principle!Birkhoff}
\label{NEW-thm:minimax-principle}\index{minimax principle}
Suppose $a$ and $b$ are regular values of a $C^2$  function $f$ on a
manifold $M$ with compact pre-image $f^{-1}[a,b]\subset M$.
For $s\in[a,b]$ consider the map of pairs
$j^s:(M^b,M^a)\to (M^b,M^s)$ induced by inclusion.
Then every non-trivial relative singular homology\footnote{
  One can replace integral singular homology by any homotopy invariant functor,
  for instance, by the homotopy functor $\pi_*$ or by the equivariant
  homology functor $\Ho_*^G$.
  }
class $\mbf{c}\in\Ho_k(M^b,M^a)\setminus\{0\}$ gives rise to a
critical value of $f$. More precisely, the three infima\footnote{
  In (\ref{NEW-eq:minimax-value}) the compact set
  $\im\sigma\CAP f^{-1}[a,b]$ is the part in $f^{-1}[a,b]$
  of the union of the (compact) images of all singular simplices that
  appear in the cycle $\sigma$. Define $\max\emptyset=-\infty$.
  }
\begin{equation}\label{NEW-eq:minimax-value}
\begin{split}
     \kappa=\kappa(\mbf{c},f)
  :&=\inf\{s\in[a,b]\mid\mbf{j^s_*}(\bc)=0\}\\
   &=\inf\{s\in[a,b]\mid\text{\rm $\bc$ comes from $\Ho_*(M^s,M^a)$}\}\\
   &=\inf_{\sigma\in\bc}\max f|_{\im \sigma\CAP f^{-1}[a,b]}\\
   &\in(a,b)
\end{split}
\end{equation}
exist and coincide and there is a critical point $x$ of $f$, non-degenerate or
not, with $f(x)=\kappa$. If $f$ is Morse on $f^{-1}[a,b]$, then more is true:
There is such $x$ whose Morse index is equal to the degree $k$ of the
relative homology class $\mbf{c}$.
\end{theorem}

Note that Theorem~\ref{NEW-thm:minimax-principle} lacks any quantitative
information, such as how many critical points to expect or, more modestly,
if different homology classes would lead to different critical levels.
For Morse functions these questions are answered
by the Morse inequalities~\cite[\S 5]{milnor:1963a}
and for general functions by the Lusternik-Schnirelmann
principle, Theorem~\ref{NEW-thm:LS-princple} in
Section~\ref{sec:LS-minmax} below, whose proof uses the thickenings
constructed in Proposition~\ref{prop:amb-thick}~(i) via Conley pairs.

Some remarks concerning the definition of $\kappa(\bc,f)$
and the exact sequence
\begin{equation*}
     \cdots
     \longrightarrow
     \Ho_*(M^s,M^a)\overset{\mbf{i^s_*}}{\longrightarrow}
     \underbrace{\Ho_*(M^b,M^a)}_{\ni\bc\not= 0}
          \overset{\mbf{j^s_*}}{\longrightarrow}
     \Ho_*(M^b,M^s)\longrightarrow
     \cdots
\end{equation*}
associated to the triple $(M^b,M^s,M^a)$ are in order. By exactness
$\mbf{j^s_*\bc=0}$ is equivalent to $\bc=\mbf{i^s_*}(\mbf{c^s})$
for some (non-trivial) class $\mbf{c^s}\in\Ho_*(M^s,M^b)$. We shall informally
abbreviate the latter situation by saying that
\textbf{$\bc$ comes from $\mbf{\Ho_*(M^s,M^a)}$}.
Observe that not only is $\Ho_*(M^b,M^b)=\{0\}$ trivial, but even
$\Ho_*(M^b,M^s)=\{0\}$ for all $s$ near $b$\,: Compactness of
$f^{-1}[a,b]\subset M$ and continuity of $f$ imply that the set of critical 
values is a compact subset of $[a,b]$, hence of $(a,b)$, as $a,b$ are
regular values. Hence any $s$ near $b$ is a regular value and the sublevel
set $M^s$ is a strong deformation retract of $M^b$ by the
regular interval Theorem~\ref{thm:no-crit_diff}.
Thus the homomorphism $\mbf{j^s_*}$ is zero near $b$
and it is the identity near $a$ by a similar argument.
Thus the infimum exists and lies in $(a,b)$.
\\
A key property is that once $\mbf{j^s_*}(\bc)$ is zero
for some value $s$ it remains zero for all larger values,
that is there are no gaps in the set $I_{\bc}$ of parameters $s$
such that $\mbf{j^s_*}(\bc)=0$. In other words,
the zero set is an interval containing the end parameter
$b$, but not the initial one $a$.

\begin{lemma}[No gaps]\label{NEW-le:no-gap-sup}
Under the assumptions of Theorem~\ref{NEW-thm:minimax-principle}
the set of all $s\in[a,b]$ for which $\mbf{j^s_*}(\bc)=0$ is zero or,
equivalently, for which $\bc$ comes from $\Ho_*(M^s,M^a)$,
is an interval $I_{\bc}$ of the form $(s_0,b]$ or $[s_0,b]$ for some $s_0>a$.
\end{lemma}

The proof is left as an exercise combining functoriality for the
inclusions $j^s=j^{\tau s}j^\tau:
(M^b,M^a)\to(M^b,M^\tau)\to(M^b,M^s)$ with the basic fact
that a homomorphism maps zero to zero.

\begin{proof}[Idea of proof of Theorem~\ref{NEW-thm:minimax-principle}.]
We already saw that the first two infima coincide and lie in $(a,b)$.
We leave the third identity in~(\ref{NEW-eq:minimax-value})
as an exercise.
It remains to show that $\kappa=\kappa(\bc,f)$ is realized as
the value of a critical point: Following~\cite{bott:1982b}
suppose $\sigma_i$ is a sequence of cycles
approximating $\kappa$ in the sense that
$\max f|_{\im \sigma_i\CAP f^{-1}[a,b]}\to\kappa$,
as $i\to\infty$. Assume by contradiction
that $\kappa$ is not a critical value.
Pick a Riemannian metric $g$ on $M$
and use the local flow $\varphi$ generated by $-\nabla^g f$,
or the corresponding level preserving local flow $\tilde\varphi_\eps$,
to push down by a fixed level difference $\eps>0$
each singular simplex appearing in $\sigma_i$.
Let $\tilde\sigma_i$ denote the corresponding sum
of the pushed down simplices.
By the \texttt{(Homotopy)} axiom of singular homology
$[\tilde\sigma_i]=[\sigma_i]=\mbf{c}$.
But $\max f|_{\im \tilde\sigma_i\CAP f^{-1}[a,b]}\to\kappa-\eps$,
as $i\to\infty$, which contradicts minimality of $\kappa$.
\\
The assertion in the Morse case holds by the 
Morse inequalities~\cite[\S 5]{milnor:1963a}.
\end{proof}

\subsection*{Subordination -- refined minimax principle}
\label{sec:LS-minmax}
\begin{definition}
Suppose $X$ is a topological space of finite cohomology type,
for instance a compact manifold. Let $R$ be a commutative ring.
Abbreviating $\Ho=\Ho(X;R)$ the cap product is a map
$\CAP:\Ho^p\times\Ho_m\to\Ho_{m-p}$; see
e.g.~\cite[VII~(12.3)]{dold:1995a}.
A non-trivial homology class $\mbf{b_1}\in\Ho_*\setminus\{0\}$
is called \textbf{subordinated} to a homology class $\mbf{b_2}$,
in symbols $\mbf{b_1}<\mbf{b_2}$,\footnote{
  Sometimes it is useful to call $\mbf{b_1}<\mbf{b_2}$ a
  \textbf{pair of subordinated homology classes}.
  }
if there is an identity of the form
$$
     \mbf{b_1}=\mbf{\omega}\CAP\mbf{b_2}
$$
for some cohomology class $\mbf{\omega}\in\Ho^{p>0}$ of
\emph{positive} degree. Note that $\mbf{b_2}$ is non-trivial and of
higher degree than $\mbf{b_1}$. Subordination is transitive. The
\textbf{$\mbf{R}$-subordination number}
$\sub_R(X)$ of $X$ is the largest integer $k\in\N_0$
such that there is a chain of subordinated classes of the form
$$
     \mbf{b_1}<\mbf{b_2}<\dots<\mbf{b_k}<\mbf{b_{k+1}}.
$$
Note that in such a chain $k$, and not $k+1$, classes are subordinated
to $b_{k+1}$. Set $\sub_R(X)=0$ in case there is no pair
$\mbf{b_1}<\mbf{b_2}$ of subordinated classes.
\end{definition}

Observe that $\cupp_R(X)=\sub_R(X)$ by the compatibility formula
$\left(\balpha\cup\bbeta\right)\cap\bc=\balpha\cap\left(\bbeta\cap\bc\right)$;
see e.g.~\cite[VII~(12.7)]{dold:1995a}.
For connected manifolds there is the obvious finiteness estimate
$\sub_R(M)\le\dim M$; cf.~(\ref{eq:cupp-dim}). 

Subordinated classes detect different critical levels, thus different critical points.
We shall formulate the result in terms of relative homology.

\begin{theorem}[The Lusternik-Schnirelmann refined minimax principle]
\label{NEW-thm:LS-princple}
Suppose $a$ and $b$ are regular values of a $C^2$ function f on a
manifold $M$ and the pre-image $f^{-1}[a,b]$ is compact. Given a pair of
\textbf{subordinated relative homology classes}\footnote{
  i.e. $\ba,\bb\not= 0$ and $\ba=\mbf{\omega}\CAP\bb$
  for some class $\mbf{\omega}\in\Ho^{p>0}(M^b)$
  where we use the cap product
  \begin{equation*}
     \Ho^p(M^b)\times\Ho_m(M^b,M^a)\overset{\cap\;}{\longrightarrow}
     \Ho_{m-p}(M^b,M^a)
  \end{equation*}
  associated to the excisive triad $(M^b;M^a,\emptyset)$;
  see e.g.~\cite[VII~(12.3)]{dold:1995a}.
  }
$\mbf{a}<\mbf{b}\in\Ho_*(M^b,M^a):=\Ho_*(M^b,M^a;R)$
for some commutative ring $R$, then the minimax critical values
from~(\ref{NEW-eq:minimax-value})
satisfy the inequality $\kappa(\mbf{a},f)\le\kappa(\mbf{b},f)$.
If all critical points of $f$ are isolated, then the inequality
$$
     \kappa(\mbf{a},f)<\kappa(\mbf{b},f)
$$
is strict and so the corresponding critical points are different.
\end{theorem}

\begin{corollary}\label{cor:LS-princple}
Suppose $R$ is a commutative ring and $f$ is a $C^2$ function on
a closed manifold $M$, then the number of critical points
$$
     \Abs{Crit f}>\sub_R(M)
$$
is strictly larger than the maximal number of consecutively
subordinated classes.
\end{corollary}

\begin{proof}
Suppose $\sub_R(M)=k$. Then there is a chain
$\mbf{b_1}<\mbf{b_2}<\dots<\mbf{b_{k+1}}$.
Now the Lusternik-Schnirelmann principle,
Theorem~\ref{NEW-thm:LS-princple},
provides corresponding critical values $c_1<c_2<\dots<c_{k+1}$.
\end{proof}

\begin{remark}[Minimal number of critical points]\label{NEW-rem:crit-f-ests}
For any $C^2$ function on a closed manifold $M$
we can now estimate the number $\abs{\Crit f}$
of critical points as follows:
If not all critical points are isolated, there are infinitely many of
them anyway. If they are isolated, the Lusternik-Schnirelmann refined
minimax principle tells that their number is strictly bounded below
by the subordination number of $M$. If all critical points
are non-degenerate, the Morse inequalities bound $\abs{\Crit f}$
from below by the dimension of the total homology of $M$ (suppose
field coefficients for simplicity). In the non-degenerate case one
has the estimates~(\ref{eq:Morse-subord-compar}).
Hence Morse theory is stronger than subordination.
\end{remark}

\begin{proof}[Proof of Theorem~\ref{NEW-thm:LS-princple}]
The identity $\mbf{a}=\mbf{\omega}\CAP\mbf{b}$ 
for some $\mbf{\omega}\in\Ho^{p>0}(M^b)$ shows that
the assumed non-triviality of $\mbf{a}$ implies non-triviality
of $\mbf{b}$. So both minimax values are defined and the
weak inequality $\kappa(\mbf{a},f)\le\kappa(\mbf{b},f)$ follows
by definition~(\ref{NEW-eq:minimax-value})
of $\kappa$ and the functoriality property\footnote{
  Note that since $j^s:(M^b,\emptyset)\to (M^b,\emptyset)$
  is the identity, it holds that $\mbf{{j^s}^*}\mbf{\omega}=\mbf{\omega}$.
  The final relative cap product
  is the one associated to the excisive triad $(M^b;M^s,\emptyset)$.
  }
$$
     \mbf{j^s_*}(\mbf{a})
     =\mbf{j^s_*}(\mbf{\omega}\CAP\mbf{b})
     =\mbf{j^s_*}(\mbf{{j^s}^*\omega}\CAP\mbf{b})
     =\mbf{\omega}\CAP \mbf{j^s_*}(\mbf{b})
$$
of the (relative) cap product under the inclusion induced triad maps
$$
     j^s:(M^b;M^a,\emptyset)\to(M^b;M^s,\emptyset),\quad s\in[a,b].
$$
For functoriality see e.g.~\cite[VII~12.6]{dold:1995a}
which applies since $(M^b;M^a,\emptyset)$ and $(M^b;M^s,\emptyset)$
are excisive triads by~\cite[III~8.1~(a)]{dold:1995a}.

Assume that all critical points of $f$ are isolated
in order to prove the strict inequality
$c:=\kappa(\mbf{a},f)<\kappa(\mbf{b},f)=:C$
of the two critical values $c,C\in(a,b)$ of $f$.
Since $f^{-1}[a,b]$ is compact and all critical points are isolated there are
only finitely many of them, thus there is an $\eps>0$ such that
the interval $[c-\eps,c+\eps]$ is a subset of $(a,b)$ and contains no critical
values other than $c$ itself.
We prove below that the projected class
$\mbf{b^{+}}:=\mbf{j_*^{{c+\eps}}}(\mbf{b})\in
\Ho_*(M^b,M^{{c+\eps}})$ is still non-trivial. Thus $\mbf{b^{+}}$
provides via~(\ref{NEW-eq:minimax-value}) the critical value
$$
     C^\prime:=\kappa(\mbf{b^{+}},f)\in({c+\eps},b)\subset(c,b)\subset(a,b).
$$
Thus $c<C^\prime$. But $C^\prime=C$ by the very
definition~(\ref{NEW-eq:minimax-value})
of $\kappa$ together with functoriality $\mbf{j_*^s}\mbf{j_*^{{c+\eps}}}
=\mbf{(j^s j^{{c+\eps}})_*}$. It also enters that, although the infimum
$C^\prime$ arises from the subset $[c+\eps,b]$ of the set $[a,b]$
used to obtain the infimum $C$, the missing points
are irrelevant since the zero condition for $\mbf{j^s_*}$ is not
satisfied at $s=c+\eps$. Thus by the no-gap Lemma~\ref{NEW-le:no-gap-sup}
the zero condition holds in both cases precisely on one and the same interval
that extends from some $s_0\in(c+\eps,b)$ all the way to and including $b$.

It remains to prove non-triviality $\mbf{b^{+}}\not= 0$,
say by contradiction.
For each critical point $x_i$ on level $c$ pick, according to
Proposition~\ref{prop:amb-thick}~(i), an open thickening $\Ww_i$
of the unstable manifold $W_i$ of $x_i$ in such a way that
the thickenings are pairwise disjoint.
Let $\Ww$ by the union of the chosen thickenings.
Consider the cohomology exact sequence associated to the inclusion
induced maps $J\circ I:\Ww\to M^b\to (M^b,\Ww)$ and note that
the restriction class $\mbf{I^*\omega}=0$ is trivial,
because $\Ww$ contracts to the critical \emph{points} on level $c$,
but the degree $p>0$ of $\mbf{\omega}$ is positive.
Thus by exactness of the sequence the class $\mbf{\omega}=\mbf{J^*\Omega}$
has a representative $\mbf{\Omega}$ coming from
$\Ho^p(M^b,\Ww)$.
\newline
Consider the inclusion induced map between excisive\footnote{
  Both triads are excisive by~\cite[III~8.1~(d)]{dold:1995a}
  since $X_1$ and $X_2$ are open in $X_1\CUP X_2$.
  } 
triads given by
$$
     f:(M^b;M^a,\emptyset)\to(M^b;M^{c+\eps},\Ww)
$$
and note that $\mbf{f^*\Omega}=\mbf{J^*\Omega}=\mbf{\omega}$ and
that the maps $\mbf{f_*}$ and $\mbf{j^{c+\eps}_*}$ coincide on $\Ho_*(M^b,M^a)$,
hence on $\bb$ and on $\ba$.
Together with functoriality, see e.g.~\cite[VII~12.6]{dold:1995a},
we get that
\begin{equation}\label{NEW-eq:cap-contra-proj}
\begin{split}
     \mbf{j^{c+\eps}_*}(\ba)
   &=\mbf{f_*}(\ba)\\
   &=\mbf{f_*}\left(\mbf{\omega}\CAP\bb\right)\\
   &=\mbf{f_*}\left(\mbf{f^*}\mbf{\Omega}\CAP\bb\right)\\
   &=\mbf{\Omega}\CAP \mbf{f_*}\bb\\
   &=\mbf{\Omega}\CAP\mbf{b^+}
     \in\Ho_*(M^b,M^{c-\eps})
\end{split}
\end{equation}
where the last cap product
\begin{equation*}
     \Ho^p(M^b,\Ww)\times
        \Ho_*(M^b,\underbrace{\Ww\cup M^{c-\eps}}_{\sim M^{c+\eps}})
     \overset{\cap\;}{\longrightarrow}
     \Ho_*(M^b,M^{c-\eps})
\end{equation*}
is the one associated to the excisive\footnote{
  The triad is excisive by~\cite[III~8.1~(a)]{dold:1995a} for
  $X_1=\Ww$ and $X_2=M^{c-\eps}$.
  }
triad $(M^b;\Ww,M^{c+\eps})$; cf.~\cite[VII~12.3]{dold:1995a}.
Obviously a key step is the homotopy equivalence
$\sim$ due to the fact that $M^{c-\eps}\cup\Ww$ is a deformation retract
of $M^{c+\eps}$. The latter follows from an analogue for isolated critical points
of the cell attachment theorem~\cite[Thm.\,3.1]{milnor:1963a}
(which requires non-degenerate critical points);
the way we defined $\Ww$ using Conley blocks with clear cut
entrance loci helps nicely.
The analogue is called the \textbf{deformation Theorem}
and it is due to Palais~\cite[Thm.~5.11]{Palais:1966b}.
\\
Now assume by contradiction that $\mbf{b^+}=0$.
Hence by~(\ref{NEW-eq:cap-contra-proj}) the projection $\mbf{j_*^{c+\eps}a}$
of the class $\ba$ to $\Ho_*(M^b,M^{c+\eps})$ vanishes even in
$\Ho_*(M^b,M^{c-\eps})$, that is $\ba=0$ in $\Ho_*(M^b,M^{c-\eps})$
or likewise $\mbf{j_*^{c-\eps}a}=0$.
So by the no-gaps Lemma~\ref{NEW-le:no-gap-sup}
we get $c=\kappa(\mbf{a},f)\le c-\eps$. Contradiction.
\end{proof}

\bibliographystyle{abbrv}
\bibliography{$HOME/Dropbox/0-Libraries+app-data/Bibdesk-BibFiles/library_math}{}

\begin{thebibliography}{10}

\bibitem{Birkhoff:1935b}
G.~D. Birkhoff and M.~R. Hestenes.
\newblock Generalized minimax principle in the calculus of variations.
\newblock {\em Duke Math. J.}, 1(4):413--432, 1935.

\bibitem{Bismut:1992a}
J.-M. Bismut and W.~Zhang.
\newblock An extension of a theorem by {C}heeger and {M}\"uller.
\newblock {\em Ast\'erisque}, 205:235, 1992.
\newblock With an appendix by Fran{\c{c}}ois Laudenbach.

\bibitem{bott:1982b}
R.~Bott.
\newblock Lectures on {M}orse theory, old and new.
\newblock {\em Bull. Amer. Math. Soc. (N.S.)}, 7(2):331--358, 1982.

\bibitem{conley:1978a}
C.~C. Conley.
\newblock {\em Isolated invariant sets and the {M}orse index}, volume~38 of
  {\em CBMS Regional Conference Series in Mathematics}.
\newblock American Mathematical Society, Providence, R.I., 1978.

\bibitem{cornea:2003a}
O.~Cornea, G.~Lupton, J.~Oprea, and D.~Tanr{\'e}.
\newblock {\em Lusternik-{S}chnirelmann category}, volume 103 of {\em
  Mathematical Surveys and Monographs}.
\newblock American Mathematical Society, Providence, RI, 2003.

\bibitem{dold:1995a}
A.~Dold.
\newblock {\em Lectures on algebraic topology}.
\newblock Classics in Mathematics. Springer-Verlag, Berlin, 1995.
\newblock Reprint of the 1972 edition.

\bibitem{Robbin:2002}
JWR.
\newblock {Lusternik Schnirelman and Cup Length, last accessed 11/01/2017 on}
  \href{http://www.math.wisc.edu/~robbin/761dir/cuplength.pdf}{Webpage},
  preprint, December 2002.

\bibitem{Kalmbach:1975a}
G.~Kalmbach.
\newblock On some results in {M}orse theory.
\newblock {\em Canad. J. Math.}, 27:88--105, 1975.

\bibitem{2016arXiv161007509K}
H.~C. {King}.
\newblock {Morse Cells}.
\newblock {\em ArXiv e-prints}, Oct. 2016.

\bibitem{Lusternik:1934a}
L.~{Lusternik} and L.~{Schnirelmann}.
\newblock {M\'ethodes topologiques dans les probl\`emes variationnels. I. Pt.
  Espaces \`a un nombre fini de dimensions. Traduit du russe par J.
  Kravtchenko.}
\newblock {Paris: Hermann \& Cie. 51 S., 5 Fig.}, 1934.

\bibitem{Majer:2006b}
P.~Majer and J.~Weber.
\newblock {Private communication during research visit of P. Majer at
  HU~Berlin}.
\newblock {B}erlin, 24~October -- 3~November 2006.

\bibitem{Majer:2015a}
P.~Majer and J.~Weber.
\newblock {Private communication during research visit of P. Majer at UNICAMP}.
\newblock Campinas, 17~January -- 27~February 2015.

\bibitem{milnor:1963a}
J.~Milnor.
\newblock {\em Morse theory}.
\newblock Based on lecture notes by M. Spivak and R. Wells. Annals of
  Mathematics Studies, No. 51. Princeton University Press, Princeton, N.J.,
  1963.

\bibitem{Palais:1966b}
R.~S. Palais.
\newblock Lusternik-{S}chnirelman theory on {B}anach manifolds.
\newblock {\em Topology}, 5:115--132, 1966.

\bibitem{palis:1982a}
J.~Palis, Jr. and W.~de~Melo.
\newblock {\em Geometric theory of dynamical systems}.
\newblock Springer-Verlag, New York, 1982.
\newblock An introduction, Translated from the Portuguese by A. K. Manning.

\bibitem{2011arXiv1102.2838Q}
L.~Qin.
\newblock {An application of topological equivalence to Morse theory}.
\newblock {\em ArXiv e-prints}, Feb. 2011.

\bibitem{salamon:1990a}
D.~Salamon.
\newblock Morse theory, the {C}onley index and {F}loer homology.
\newblock {\em Bull. London Math. Soc.}, 22(2):113--140, 1990.

\bibitem{weber:2014b}
J.~Weber.
\newblock The {B}ackward $\lambda$-{L}emma and {M}orse {F}iltrations.
\newblock In {\em Analysis and topology in nonlinear differential equations},
  volume~85 of {\em Progr. Nonlinear Differential Equations Appl.}, pages
  457--466. Birkh\"auser/Springer, Cham, 2014.

\bibitem{Weber:2015c}
J.~Weber.
\newblock Contraction method and {L}ambda-{L}emma.
\newblock {\em S{\~a}o Paulo Journal of Mathematical Sciences}, 9(2):263--298,
  2015.

\bibitem{Weber:2014e}
J.~Weber.
\newblock {Classical Morse theory revisited -- I Backward $\lambda$-Lemma and
  homotopy type}.
\newblock {\em Topol. Methods Nonlinear Anal.}, 47(2):641--646, 2016.

\bibitem{weber:2014c}
J.~Weber.
\newblock {S}table foliations and semi-flow {M}orse homology.
\newblock {\em Ann. Sc. Norm. Super. Pisa Cl. Sci. (5)}, Vol. XVII(3):853--909,
  2017.

\bibitem{zehnder:2010a}
E.~Zehnder.
\newblock {\em Lectures on dynamical systems}.
\newblock EMS Textbooks in Mathematics. European Mathematical Society (EMS),
  Z{\"u}rich, 2010.

\end{thebibliography}

\end{document}